\documentclass[a4paper, 11pt]{preprint}
\usepackage{amsthm, amsmath, amssymb, mathtools}
\usepackage[cal=boondoxo]{mathalfa}
\usepackage[full]{textcomp}
\usepackage[osf]{newtxtext}
\usepackage{mhequ}
\usepackage{mathrsfs}
\usepackage{microtype}
\usepackage{tikz}
\usepackage{enumitem}
\usepackage{comment}
\usepackage{orcidlink}
\usepackage{cprotect}
\usepackage{wasysym}
\usepackage{centernot}

\usepackage{hyperref}

\DeclareSymbolFont{timesoperators}{T1}{ptm}{m}{n}
\SetSymbolFont{timesoperators}{bold}{T1}{ptm}{b}{n}

\def\emptyset{\mathord{\centernot{\text{\rm$\Circle$}}}}

\usetikzlibrary{calc}
\usetikzlibrary{shapes}
\usetikzlibrary{decorations}
\usetikzlibrary{decorations.markings}
\usetikzlibrary{decorations.pathmorphing}

\newcommand{\eqdef}{\stackrel{\mbox{\rm\tiny def}}{=}}

\makeatletter
\renewcommand{\operator@font}{\mathgroup\symtimesoperators}
\makeatother

\usetikzlibrary{external}

\colorlet{darkred}{red!90!black}
\colorlet{darkblue}{blue!90!black}
\colorlet{lightblue}{blue!50}

\newtheorem{theorem}{\bf  {Theorem}}
\newtheorem{proposition}[theorem]{\bf {Proposition}}
\newtheorem{lemma}[theorem]{\bf Lemma}

\newtheorem{corollary}[theorem]{\bf Corollary}
\newtheorem{remark}[theorem]{\bf Remark}

\numberwithin{theorem}{section}

\DeclarePairedDelimiter\abs\lvert\rvert

\colorlet{symbols}{black}
\colorlet{testcolor}{green!60!black}

\tikzset{
	root/.style={circle,fill=testcolor,inner sep=0pt, minimum size=2mm},
	broot/.style={circle,fill=gray,inner sep=0pt, minimum size=2mm},
	dot/.style={circle,fill=black,inner sep=0pt, minimum size=1mm},
		reddot/.style={circle,fill=red,inner sep=0pt, minimum size=1mm},
			bluedot/.style={circle,fill=blue,inner sep=0pt, minimum size=1mm},
	eps/.style={circle,fill=white,draw=symbols,inner sep=0pt,minimum size=0.8mm},
	int/.style={circle,fill=black,draw=black,inner sep=0pt,minimum size=0.7mm},
	var/.style={circle,fill=black!10,draw=black,inner sep=0pt, minimum size=2mm},
	dotred/.style={circle,fill=black!50,inner sep=0pt, minimum size=2mm},
	generic/.style={semithick,shorten >=1pt,shorten <=1pt},
	dist/.style={ultra thick,draw=testcolor,shorten >=1pt,shorten <=1pt},
	testfcn/.style={ultra thick,testcolor,shorten >=1pt,shorten <=1pt,->},
	testfcnx/.style={ultra thick,testcolor,shorten >=1pt,shorten <=1pt,<-,
		postaction={decorate,decoration={markings,mark=at position 0.6 with {\drawx}}}},
	keps/.style={semithick,shorten >=1pt,shorten <=1pt,densely dashed,->},
	kprimex/.style={semithick,shorten >=1pt,shorten <=1pt,densely dashed,->,
		postaction={decorate,decoration={markings,mark=at position 0.4 with {\drawx}}}},
	kernel/.style={semithick,shorten >=1pt,shorten <=1pt,->},
	multx/.style={shorten >=1pt,shorten <=1pt,
		postaction={decorate,decoration={markings,mark=at position 0.5 with {\drawx}}}},
	kernelBig/.style={semithick,shorten >=1pt,shorten <=1pt,decorate, decoration={zigzag,amplitude=1.5pt,segment length = 3pt,pre length=2pt,post length=2pt}},
	kernelBig2/.style={semithick,shorten >=1pt,shorten <=1pt,decorate,->, decoration={snake,amplitude=1.5pt,segment length = 3pt,pre length=2pt,post length=5pt}},
	rho/.style={dotted,semithick,shorten >=1pt,shorten <=1pt},
	renorm/.style={shape=circle,fill=white,inner sep=1pt},
	labl/.style={shape=rectangle,fill=white,inner sep=1pt},
	H/.style={circle,fill=blue!10,draw=symbols,inner sep=0pt,minimum size=1.3mm},
	xi/.style={regular polygon sides=4,fill=red!30,draw=symbols,inner sep=0pt,minimum size=1.1mm},
	xix/.style={crosscircle,fill=symbols!10,draw=symbols,inner sep=0pt,minimum size=1.2mm},
	xib/.style={circle,fill=symbols!10,draw=symbols,inner sep=0pt,minimum size=1.6mm},
	xibx/.style={crosscircle,fill=symbols!10,draw=symbols,inner sep=0pt,minimum size=1.6mm},
	not/.style={circle,fill=symbols,draw=symbols,inner sep=0pt,minimum size=0.5mm},
Wick/.style={rectangle, draw=blue!80,rounded corners=3pt,fill=blue!5},
	>=stealth,
	not/.style={circle,fill=symbols,draw=symbols,inner sep=0pt,minimum size=0.5mm},
kernels2/.style={very thick,segment length=12pt},
	}
\makeatletter
\def\DeclareSymbol#1#2#3{%
	\expandafter\gdef\csname MH@symb@#1\endcsname{\tikzsetnextfilename{symbol#1}%
		\tikz[baseline=#2,scale=0.15,draw=symbols,line join=round]{#3}}%
	\expandafter\gdef\csname MH@symb@#1s\endcsname{\scalebox{0.75}{\tikzsetnextfilename{symbol#1}%
			\tikz[baseline=#2,scale=0.15,draw=symbols,line join=round]{#3}}}%
	\expandafter\gdef\csname MH@symb@#1ss\endcsname{\scalebox{0.65}{\tikzsetnextfilename{symbol#1}%
			\tikz[baseline=#2,scale=0.15,draw=symbols,line join=round]{#3}}}%
}
\def\<#1>{\ifthenelse{\boolean{mmode}}{\mathchoice{\csname MH@symb@#1\endcsname}{\csname MH@symb@#1\endcsname}{\csname MH@symb@#1s\endcsname}{\csname MH@symb@#1ss\endcsname}}{\csname MH@symb@#1\endcsname}}
\makeatother

\DeclareSymbol{N1}{-1}{\draw (0,0) -- (-0.8,1.2);\draw (0,0) node[dot]{}  -- (0,0);}
\DeclareSymbol{N2}{-1}{ \draw (0,0) -- (0.8,1.2);\draw (0,0) -- (-0.8,1.2);\draw (0,0) node[dot]{}  -- (0,0);}
\DeclareSymbol{N3}{-3}{\draw (0,0) node[xi]{}  -- (0,0);}
\DeclareSymbol{N4}{-3}{\draw (0,0) node[H]{}  -- (0,0);}
\DeclareSymbol{one}{-3}{\draw (0,0) node[dot]{}  -- (0,0);}
\DeclareSymbol{N5}{-1}{\draw (-0.8,1.2) -- (0,0) node[H]{}  -- (0,0);}

\DeclareSymbol{ex}{3}{\draw (0,3) node[xi]{}  -- (-1,1.5) node[H]{}
	-- (0,0) node[dot]{} -- (1,1.5) node[H]{};}

\DeclareSymbol{xi}{-2}{\draw (0,0) node[xi,minimum size=1.6mm]{};}
\DeclareSymbol{H}{-2}{\draw (0,0) node[H,minimum size=1.8mm]{};}

\def\K{\mathfrak{K}}

\def\CA{\mathcal{A}}
\def\CB{\mathcal{B}}
\def\CC{\mathcal{C}}
\def\CT{\mathcal{T}}
\def\CG{\mathcal{G}}
\def\CD{\mathcal{D}}

\def\CI{\mathcal{I}}
\def\CK{\mathcal{K}}

\def\CR{\mathcal{R}}

\def\CJ{\mathcal{J}}
\def\CM{\mathcal{M}}
\def\CS{\mathcal{S}}

\def\CP{\mathcal{P}}

\def\cC{\mathscr{C}}

\def\R{\mathbf{R}}
\let\phi\varphi
\def\scal#1{\langle #1\rangle}
\def\${|\!|\!|}

\def\Vec{\operatorname{span}}
\let\eps\varepsilon

\def\PPi{\boldsymbol{\Pi}}

\def\one{\mathbf{1}}

\let\div\undefined
\DeclareMathOperator{\div}{div}
\DeclareMathOperator{\sign}{sign}
\def\sol{{\mathrm{sol}}}

\def\E{\mathbf{E}}
\def\X{\mathbb{X}}

\def\B{\mathbb{B}}

\def\N{\mathbf{N}}
\def\Z{\mathbf{Z}}
\def\f#1#2{\textstyle{\frac{#1}{#2}}}
\let\d\partial

\let\up\uparrow
\let\down\downarrow

\def\restr{\mathbin{\upharpoonright}}

\def\dash{\leavevmode\unskip\kern0.18em--\penalty\exhyphenpenalty\kern0.18em}
\def\slash{\leavevmode\unskip\kern0.15em/\penalty\exhyphenpenalty\kern0.15em}

\makeatletter

\DeclareRobustCommand{\TitleEquation}[2]{\texorpdfstring{\StrLeft{\f@series}{1}[\@firstchar]$\if%
		b\@firstchar\boldsymbol{#1}\else#1\fi$}{#2}}

\makeatother

\begin{document}	
	\title{Renormalisation in the presence of variance blowup}
	\author{Martin Hairer$^{1,2}$ \orcidlink{0000-0002-2141-6561}}
	\institute{École Polytechnique Fédérale de Lausanne, 1015 Lausanne, Switzerland 
	\and Imperial College London, London SW7 2AZ, United Kingdom\\
		\email{martin.hairer@epfl.ch}}
	
	\maketitle
	
	\begin{abstract}
	We show that if one drives the KPZ equation by the derivative
	of a space-time white noise smoothened out at scale $\eps \ll 1$ and multiplied by
	$\eps^{3/4}$ then,
	as $\eps \to 0$, solutions converge to the Cole--Hopf solutions to the 
	KPZ equation driven by space-time white noise.
	
	In the same vein, we also show that if one drives an SDE by fractional Brownian motion with
	Hurst parameter $H < 1/4$, smoothened out at scale $\eps \ll 1$ and multiplied by
	$\eps^{1/4-H}$ then, as $\eps \to 0$, solutions converge to an SDE 
	driven by white noise. 
	The mechanism giving rise to both results is the same, but the proof techniques
	differ substantially. 

\vspace{1em}


	\end{abstract}
	
	\setcounter{tocdepth}{2}
	\tableofcontents
	
\section{Introduction}

The past two decades have seen a wealth of results in the general area of 
singular stochastic ODEs / parabolic PDEs with the advent of rough paths \cite{Terry,Max,Book},
regularity structures \cite{Hai14}, paracontrolled calculus \cite{GIP15},
the rigorous implementation of Wilson's RG and the Polchinski's flow equation
\cite{Antti,Pawel}, etc. A common feature of all of these works (as well as their further
refinements) is that their domain of validity is cleanly delimited by two main
conditions. The first (and physically most relevant) condition is one of
\textit{scaling subcriticality} or, in QFT language, \textit{superrenormalisability}. 
This condition, roughly speaking, guarantees that solutions to the equation of
interest locally look like Gaussian processes at small scales. When this condition
fails to be satisfied, i.e.\ in the supercritical case, one typically expects 
any natural approximation to the problem
to converge to  a ``trivial'' (often Gaussian) limit as its approximation scale $\eps$
is sent to zero. See for example \cite{Fröhlich,Aizenman}
in the context of  the 
$\Phi^4$ models as well as \cite{KPZ3,MR3851818,MR4087492} in the context of the KPZ
equation. At the boundary, i.e.\ when the model is \textit{scaling critical}, 
interesting phenomena can occur but depend on much finer properties of the model.
For example the $\Phi^4$ model turns out to be trivial at its critical dimension $4$
\cite{ADC21} while the KPZ equation appears to exhibit non-trivial behaviour 
\cite{KPZ2D1,KPZ2D2,KPZ2D3}.

There is however a second somewhat less widely known condition that appears in the
above mentioned works as a consequence of the following. These works all exploit the fact
that the solution can locally be described by a linear combination of certain
multilinear expressions in the driving noise(s). When interpreted ``naïvely'' some
of these expressions blow up as $\eps \to 0$ and need to be renormalised in order to
yield finite limits. This renormalisation procedure can be thought of as a 
way of recentering these expressions in a nonlinear way that is compatible with their
algebraic structure. The recentered expressions have vanishing expectations and it is
a generic fact that their variances are better behaved and typically converge to
a finite limit. In certain cases however (and this only happens when considering equations
driven by a noise that is less regular than the corresponding space-time white noise)
it may happen that, despite being scaling subcritical, the problem is such that one
of these variances also diverges as $\eps \to 0$. Maybe the most prominent example of
this phenomenon is given by SDEs driven by independent fractional Brownian motions $B_i$
with Hurst parameter 
$H < \f14$. In this case, the Lévy area of natural approximations to these driving 
processes diverges even though, at the analytical level, the theory of rough paths
would be applicable. It follows from general principles that it is possible to
construct a rough path above $B$ \cite{Extension,Samy,Unterberger}, but 
these constructions are non-canonical (not even modulo finitely many parameters)
and have pathological properties. In particular, the resulting solution flow
$\Phi_{s,t}$ is not measurable with respect to the $\sigma$-algebra generated by
the increments of $B$ over $[s,t]$.

In this article, we study two examples exploring the situation where 
subcriticality holds but the variance of one of the stochastic objects associated
to our problem diverges. These are given by the KPZ equation driven by (approximations to)
the spatial derivative of space-time white noise and SDEs driven by fractional Brownian
motion with $H \le \f14$. We will see that in both cases the limit is given by
an equation in the same class but driven by white noise. See the discussion at the
end of this introduction for more details regarding what one would expect in general.

\subsection{The KPZ equation}
	
It is by now well known that the KPZ equation can be derived from a large class of 
interface fluctuation models in the weakly asymmetric regime. The first mathematically
rigorous proof of this fact dates back to the seminal work by Bertini and Giacomin \cite{BG99}
on fluctuations of the weakly asymmetric simple exclusion process which exploited a number of
very specific features of the latter, in particular that it behaves ``nicely'' under
the Cole--Hopf transform which is used there to interpret the KPZ equation itself.

One of the simplest convergence results \cite{Hai13,Hai14,HaoCLT} states that if $u$ solves the PDE
\begin{equ}[e:weakAsym]
\d_t u = \d_x^2 u + \sqrt \eps (\d_x u)^2 + \bar \eta\;,
\end{equ}
for $\bar \eta$ a stationary random field with good enough mixing properties, then there exists a constant $C_\eps$ such that
$h_\eps(t,x) = \sqrt \eps u(t/\eps^2,x/\eps) - C_\eps t$ converges as $\eps \to 0$ to the KPZ equation. 
As one would expect, the variance of the resulting space-time white noise driving the KPZ equation is
given by the space-time integral of the variance of $\bar \eta$. Furthermore, the leading order of the constant 
$C_\eps$ is of the form $c\eps^{-1}$.

This leads naturally (at least from a mathematical perspective) to the question of what happens when the 
variance of $\bar \eta$ has vanishing integral. For example, does \eqref{e:weakAsym} admit a non-trivial scaling limit
when $\bar \eta$ is replaced by $\d_x \bar \eta$? Under the same scaling and recentering as above, 
we know that solutions simply converge
to the solution of the deterministic equation $\d_t h = \d_x^2 h + (\d_x h)^2$ but is there a scaling 
under which they converge to a non-trivial stochastic process? It turns out that this is indeed the case provided that we
increase the strength of the nonlinearity and consider a slightly different scaling. 

One aim of this article is to show that, if we consider solutions to the equation
\begin{equ}[e:strongAsym]
\d_t u = \d_x^2 u +  \eps^{1/4} (\d_x u)^2 + \d_x\bar \eta\;,
\end{equ}
and set $h_\eps(t,x) = \eps^{1/4} u(t/\eps^2,x/\eps) - C_\eps t$ (note the difference in scaling from
the one that keeps the stochastic heat equation invariant!), then a suitable choice of $C_\eps$ leads
again to the convergence of $h_\eps$ to solutions to the KPZ equation. The constant itself also
behaves slightly differently: to leading order it now scales like $C_\eps \sim \eps^{-3/2}$.

After rescaling, \eqref{e:strongAsym} can be written as
\begin{equ}[e:basic]
\d_t h_\eps = \d_x^2 h_\eps +  (\d_x h_\eps)^2 + \eps^{3/4} \d_x\eta_\eps - C_\eps\;,
\end{equ}
where $\eta_\eps$ is scaled so that it converges to a white noise.
In order to simplify notations and to avoid technical complications we fix a space-time white noise $\eta$ on
$S^1 \times \R$ and we assume that there exists a smooth function $\rho \colon \R^2 \to \R$ supported in the 
unit ball such that $\eta_\eps = \rho_\eps \star \eta$ where $\rho_\eps(z) = \eps^{-3}\rho(\eps^{-1}z)$ and
$\star$ denotes space-time convolution. We then consider \eqref{e:basic} as a Cauchy problem in 
$\CC^\alpha(S^1)$ for some $\alpha \in (0,1/2)$, in particular we restrict ourselves to a compact spatial domain
with periodic boundary conditions. With these notations at hand, we have the following result,
the proof of which will be provided in Section~\ref{sec:convergenceKPZ}.

\begin{theorem}\label{theo:main:KPZ}
Let $\alpha \in (0,1/2)$, $\beta \in (0,\alpha\wedge 1/4)$, and let $h_0 \in \CC^\alpha(S^1)$. 
Then, there exists $\sigma > 0$ and a choice of constants $C_\eps$ such that the solution $h_\eps$ to \eqref{e:basic}
converges in law in $\CC(\R_+, \CC^\beta(S^1))$ to the solution to the Cole--Hopf solution to the KPZ
equation driven by space-time white noise with variance $\sigma^2$.
\end{theorem}

A surprising feature of this result is that the limiting KPZ equation is driven by 
\textit{space-time white noise} while $\d_x \eta$, suitably rescaled, converges to the spatial 
derivative of such a noise. So where does this driving noise come from? The basic observation is
that, writing $P$ for the heat kernel and $\star$ for space-time convolution, 
$(\d_x P \star \d_x\eta)^2$ suitably rescaled and recentered converges weakly to a space-time white noise
and it is this noise that drives the limiting KPZ equation. Furthermore, asymptotically as
$\eps \to 0$, this driving noise is actually \textit{independent} of the original field $\eta$ in the 
sense of stable convergence \cite{JacodLimit, StableBook}.

\subsection{SDEs driven by fractional noise}

We also show that a similar phenomenon arises when considering SDEs driven by fractional
Brownian motion with Hurst parameter $H \le 1/4$. Let $\{B_i\}_{i=1}^m$ be i.i.d.\ copies of a fractional Brownian motion with some fixed Hurst parameter $H \le 1/4$ and let $B_i^\eps$ be their convolution
with a mollifier at scale $\eps$. Let furthermore $V_i$ be $m$ vector fields on 
$\R^d$ (this could be replaced by a manifold) and let $x^\eps$ be the solution to the 
controlled differential equation
\begin{equ}[e:SDEfbm]
dx^\eps = C_\eps \sum_{i=1}^m V_i(x^\eps)\,dB_i^\eps\;,
\end{equ}
where we set $C_\eps = \eps^{1/4-H}$ when $H < 1/4$ and $C_\eps = \abs{\log \eps}^{-1/2}$
when $H = 1/4$. We then have the following convergence result.

\begin{theorem}\label{thm:mainfbm}
There exists a finite constant $\sigma>0$ and independent standard Wiener processes $W_{ij}$
such that $x^\eps$ converges in law to the diffusion process $x$ solving
\begin{equ}[e:limitDiffusion]
dx = \frac\sigma2\sum_{i<j} [V_i,V_j](x)\circ dW_{ij}\;,
\end{equ}
where ${}\circ dW$ denotes Stratonovich integration against $W$ and $[V,\bar V]$ denotes the Lie bracket
of two vector fields $V$, $\bar V$.
\end{theorem}

\begin{remark}
Convergence takes place uniformly in time
until the first time solutions to \eqref{e:limitDiffusion} get large,
in particular we do not need to assume global well-posedness. See Theorem~\ref{thm:mainRigorous} 
below for a precise formulation taking topologies into account.
\end{remark}

\begin{remark}
It was shown in \cite{unterberger2008central} that the Lévy area of two independent 
fractional Brownian motions with $H < \f14$ converges to an independent Wiener process,
which strongly hints at Theorem~\ref{thm:mainfbm} but is far from sufficient to prove it.
A result with a somewhat similar flavour (a Brownian motion independent of the
original fractional Brownian motion showing up as a second-order process) was
previously obtained when $H = 1/4$ in \cite{Nourdin,Chris}. 
\end{remark}

The proof of these two results suggests the following general mechanism. Take a class of SPDEs (or SDEs) 
driven by noise of regularity $\alpha < 0$. 
What we mean here by a ``class'' is that we fix a rule in the 
sense of \cite{BHZ} as well as an ambient space(-time) dimension and consider all
equations that can be formulated within the corresponding regularity structure.
One the one hand, there then exists some $\alpha_c$ such that
the problem is locally subcritical (in the sense of \cite{Hai14}) when $\alpha > \alpha_c$ and fails to 
be so when $\alpha < \alpha_c$. On the other hand, there exists $\alpha_v$ such that 
the third condition of \cite[Thm~2.15]{ChHa16} is satisfied for $\alpha > \alpha_v$ and fails for $\alpha < \alpha_v$. This condition essentially states that the homogeneities of all ``composite'' symbols
(i.e.\ those containing at least two instances of a noise) are greater than that of space-time white
noise.
This threshold can be understood intuitively by the fact that the variance of white noise 
is a delta-function which, in dimension $d$, behaves like $|x|^{-d}$ in terms of scaling, thus
representing the boundary between integrability and non-integrability of the variance.
When it fails, the \textit{variance} of some of the stochastic objects
relevant in the description of the solution blows up, which cannot be controlled by the addition
of counterterms in the equation (these can only counteract the divergence of \textit{expectations}).

What our result shows is that if we are in a situation where $\alpha_c < \alpha_v$ and we choose as driving noise
a mollification at scale $\eps$ of a noise that is self-similar with some exponent $\beta < \alpha_v$, then we
obtain a non-trivial limit driven by (space-time) white noise within the same class of 
equations, provided that we multiply our driving noise with $\eps^{\alpha_v - \beta}$ (or
a suitable negative power of $\abs{\log\eps}$ if $\beta = \alpha_v$).
Indeed in the example of the KPZ equation, one has $\alpha_c = -2$ and $\alpha_v = - 7/4$. Our driving noise is a
mollification of the derivative of space-time white noise which is self-similar with exponent
$-5/2$, so that this heuristic suggests that we multiply it by $\eps^{3/4}$ (since $3/4 = -7/4+5/2$)
in order to obtain a non-trivial limit.
In the case of SDEs \slash RDEs one has $\alpha_c = -1$ and $\alpha_v = -3/4$ while our driving noise has
self-similarity exponent $\beta = H-1$ (the driving noise is the derivative of our fractional Brownian motion), 
which again leads to the correct prediction.

Note that there is in general no particular ordering between $\alpha_c$ and $\alpha_v$, especially since the latter
is dimension-dependent while the former is not. In the case of 
the $\Phi^4$ equation, one has for example $\alpha_c = -3$ and $\alpha_v = -(d+14)/6$ so that the phenomenon
described in this article is expected to take place in spatial dimension $d < 4$, which covers
the whole subcritical regime. 

One point that the reader should keep in mind is that, depending on how one interprets an equation, one may be
lead to descriptions of its solutions using different regularity structures which could potentially lead
to different values for $\alpha_v$. For example, when writing the Navier--Stokes nonlinearity in its
usual form as $(u\cdot\nabla)u$, one is lead to $\alpha_v = -(d+8)/4$. If however we remember that,
thanks to the divergence free condition, it can equivalently be written as $\div(u\otimes u)$,
we find that one actually has $\alpha_v = -(d+10)/4$.
In other words, while $\alpha_v$ determines the noise regularity at which the variance of some
stochastic object \textit{may} diverge, it does not guarantee that it \textit{will} diverge
since there could be additional problem-specific cancellations.

\subsection{Notations}

We will always use parabolic scalings so that, given $z = (t,x) \in \R^2$ and $\lambda \in \R$, we write
\begin{equ}
\lambda z = (\lambda^2 t, \lambda x)\;,\qquad |z| = \sqrt{|t|} + |x|\;.
\end{equ}
Test functions $\phi$ are always smooth, supported in the unit ball, and such that $\sup_z \abs{D^k \phi(z)} \le 1$
for all $|k| \le 10$ (say). As usual, we denote its rescaled translates by $\phi^\lambda_z(z_1) = \lambda^{-3}\phi(\lambda^{-1}(z_1-z))$.

\section{Convergence to KPZ}

Writing $h = u - \eps^{3/4} P\star  \d_x\eta_\eps$, setting 
\begin{equ}[e:defxieps]
\xi_\eps = \eps^{3/2}(\d_x P \star \d_x\eta_\eps)^2 - C_\eps\;,
\end{equ}
with $C_\eps$ such that $\E \xi_\eps \equiv 0$,
and setting
\begin{equ}
\chi_\eps = \eps^{3/4}\d_x P \star \d_x\eta_\eps = \eps^{3/4}\d_x^2 P \star \rho_\eps \star \eta\;,
\end{equ}
we can rewrite
this equation suggestively as
\begin{equ}[e:basich]
\d_t h = \d_x^2 h +  (\d_x h)^2 + 2\chi_\eps \d_x h + \xi_\eps  \;.
\end{equ}
At the intuitive level, our main result then follows from Proposition~\ref{prop:convNoise} combined
with the fact that $\chi_\eps \to 0$. This is of course far from a proof since $\eps^{3/4} \d_x\eta_\eps \to 0$, 
so essentially the same argument would suggest
that \eqref{e:strongAsym} converges to a deterministic limit!

The main reason why the formal argument works in the case of \eqref{e:basich} but fails in the case
of \eqref{e:basic} is that if we apply the methodology of \cite{Hai14,ChHa16,Algebraic,Renorm} 
to \eqref{e:basic} (using the fact that, uniformly over
$\eps$, one has $\eps^{3/4} \d_x\eta_\eps \in \CC^{-7/4^-}$ with parabolic scaling but no better), then we end up 
generating a symbol containing two instances of the noise, but with degree $-3/2^-$,
thus (just about) failing to satisfy the third condition of \cite[Thm~2.15]{ChHa16}. In fact, an analogous 
condition first appeared in the context of rough paths in \cite{CQ00} and also appears for example
in \cite[Ass.~2.1]{SGFelix} and
the second part of \cite[Ass~2.31]{SGap}.
When applying it to \eqref{e:basich} on the other hand, the ``worst'' symbol with more 
than one noise that arises corresponds to $\chi_\eps \cdot (\d_xP\star \xi_\eps)$
and is of degree $-5/4^- >-3/2$.

\subsection{Convergence of the noise}

We use graphical notations similar to those of \cite{Hai14,HaoCLT}. In particular, we use plain nodes \tikz[baseline=-0.25em] \node[dot] {}; 
for integration variables, a coloured node  \tikz \node[root] {}; for a variable taking the fixed value $0$, and 
\tikz \node[var] {};
for an integration variable at which we furthermore evaluate an instance of the white noise $\eta$. 
We group multiple white noises together like so \tikz[baseline=-0.3em] \node[Wick] {}; to indicate that their product should be interpreted
as a Wick product.

As usual, all of our variables are space-time variables. 
Regarding edges (which represent kernels evaluated at the difference
between the two variables represented by the nodes they connect), 
we write \tikz[baseline=-0.25em] \draw[rho] (0,0) -- (1,0); for the mollifier $\rho_\eps$, 
\tikz[baseline=-0.25em] \draw[testfcn,->] (0,0) -- (1,0); for a test function $\phi^\lambda(t,x) $ centred at the origin and 
rescaled at some scale $\lambda$ (if no scale is specified one takes $\lambda = 1$, different test functions
are drawn with different colors), and
\tikz[baseline=-0.25em] \draw[kernel] (0,0) -- (1,0); for the kernel
$(t,x) \mapsto \eps^{3/4} \d_x^2 P_t(x)$. In particular, this graphical notation yields
\begin{equ}[e:graphics]
\scal{\eta_\eps, \phi} = 
\begin{tikzpicture}[scale=0.35,baseline=0cm]
	\node at (0,-1)  [root] (root) {};
	\node at (-1.5,0.5)  [dot] (left) {};
	\node at (0,2) [var] (variable1) {};
	
	\draw[testfcn] (left) to  (root);	
	\draw[rho] (variable1) to (left); 
\end{tikzpicture}\;,\quad
\scal{\chi_\eps, \phi} = 
\begin{tikzpicture}[scale=0.35,baseline=0cm]
	\node at (0,-1)  [root] (root) {};
	\node at (-1.5,0)  [dot] (left) {};
	\node at (-1.5,1.5)  [dot] (left1) {};
    
	\node at (0,2.5)  [var] (variable1) {};

	\draw[testfcn] (left) to  (root);	
	\draw[kernel] (left1) to (left);
	\draw[rho] (variable1) to (left1); 
\end{tikzpicture}\;,\quad
\scal{\xi_\eps, \phi} = 
\begin{tikzpicture}[scale=0.35,baseline=0cm]
    \node [Wick, minimum width=1cm, minimum height=0.5cm, anchor=center] at (0,2.5) {};

	\node at (0,-1.5)  [root] (root) {};
	\node at (0,0)  [dot] (bot) {};
	\node at (-1.5,1.5)  [dot] (left1) {};

	\node at (1.5,1.5)  [dot] (right1) {};

	\node at (-.75,2.5)  [var] (variable1) {};
	\node at (.75,2.5)  [var] (variable2) {};
	
	\draw[testfcn] (bot) to  (root);	
	\draw[kernel] (left1) to (bot);
	\draw[kernel] (right1) to (bot);
	\draw[rho] (variable1) to (left1); 
	\draw[rho] (variable2) to (right1); 
\end{tikzpicture}
\end{equ}
An important role will be played by the covariance $K_\eps$ of $\chi_\eps$
as well as the correlation between $\eta_\eps$ and $\chi_\eps$, so we introduce the 
shorthand graphical notations
\begin{equ}
K_\eps = \tikz[baseline=-0.25em] \draw[kernelBig] (0,0) -- (1,0); =
\begin{tikzpicture}[scale=0.35,baseline=-0.1cm]
	\node at (-2,0)  [dot] (left) {};
	\node at (0,0)  [dot] (mid) {};
	\node at (2,0)  [dot] (right) {};
    
	\draw[kernel] (left) to (-4,0);
	\draw[kernel] (right) to (4,0);
	\draw[rho] (left) to (mid); 
	\draw[rho] (right) to (mid); 
\end{tikzpicture}\;,\qquad
\bar K_\eps = \tikz[baseline=-0.25em] \draw[kernelBig2] (0,0) -- (1,0); =
\begin{tikzpicture}[scale=0.35,baseline=-0.1cm]
	\node at (-2,0)  [dot] (left) {};
	\node at (0,0)  [dot] (mid) {};
    
	\draw[rho] (left) to (-4,0);
	\draw[rho] (left) to (mid); 
	\draw[kernel] (mid) to (right); 
\end{tikzpicture}\;.
\end{equ}
The following is a straightforward consequence of the scaling properties of $\rho_\eps$ and the heat kernel.

\begin{lemma}\label{lem:boundKeps}
One has $K_\eps(z) = \eps^{-3/2} K(z/\eps)$ for a smooth function $K$ satisfying the bounds
$\abs{K(z)} \lesssim (1+|z|)^{-3}$.
Similarly, one has $\bar K_\eps(z) = \eps^{-9/4} \bar K(z/\eps)$ for a smooth function $\bar K$ satisfying the bounds
$\abs{\bar K(z)} \lesssim (1+|z|)^{-3}$.
\qed
\end{lemma}

\begin{proposition}\label{prop:convNoise}
Let $\xi_\eps$ be as in \eqref{e:defxieps}. Then, there exists
$\sigma > 0$ such that $\xi_\eps$ converges stably 
in $\CC^\alpha$ for any $\alpha < -3/2$ to a space-time white noise $\xi$
with variance $\sigma^2$ independent of $\eta$.
\end{proposition}

\begin{proof}
To show stable convergence, it suffices  by \cite[Prop.~3.4~(c)]{StableBook}
to show that the pair $(\xi_\eps, \eta_\eps)$
converges weakly in $\CC^\alpha\times \CC^\alpha$ to a pair of independent white noises.

We first note that $\xi_\eps$ and $\eta_\eps$ are decorrelated and that one has
\begin{equ}
\abs{\E \scal{\phi_z^\lambda, \xi_\eps}^2}
= 2\Bigl|\int \phi_z^\lambda(z_1)\phi_z^\lambda(z_2) K_\eps^2(z_1-z_2)\,dz_1\,dz_2\Bigr|\;.
\end{equ}
Since $\phi_z^\lambda$ is supported in a parabolic ball of radius $\lambda$, this integral is 
at most of order $\lambda^3 \|\phi^\lambda\|_{L^\infty}^2 \|K_\eps^2\|_{L^1} \approx \lambda^{-3}$.
Since a similar bound holds for $\eta_\eps$, equivalence of moments for elements in a fixed Wiener
chaos combined with \cite[Thm~1.1]{Kolmogorov} yields tightness in $\CC^\alpha\times \CC^\alpha$. 

Since $K_\eps^2$ converges to some multiple of a Dirac mass, it immediately follows similarly to above that
there exists $\sigma>0$ such that 
\begin{equ}
\lim_{\eps \to 0}\E \scal{\phi, \xi_\eps}\scal{\xi_\eps,\psi}
= \sigma^2\scal{\phi,\psi}_{L^2}\;,
\end{equ}
and similarly for $\eta_\eps$, for any continuous compactly supported test functions $\phi$ and $\psi$.

It therefore remains to show that any collection of random variables of the form $\xi_\eps(\phi)$ 
and \slash or $\eta_\eps(\psi)$ converges to a jointly Gaussian limit as $\eps \to 0$.
For this, we use the Nualart--Peccati fourth moment theorem \cite{FourthMoment} which implies that this 
is the case provided that the fourth joint cumulant of any such collection of variables converges to $0$ as
$\eps \to 0$. It follows from the diagram formula \cite[Thm~7.1.3]{PeccatiTaqqu} that the joint cumulant
$\E_c \big(\scal{\phi_1,\xi_\eps},\ldots,\scal{\phi_4,\xi_\eps}\big)$
is obtained by taking four copies of the figure associated to 
$\scal{\xi_\eps, \phi}$ in \eqref{e:graphics}
and summing over all ways of pairwise contracting the white noises, subject to the
following constraints:
\begin{itemize}
\item No two noises grouped into a Wick product can be contracted.
\item The resulting diagram is connected.
\end{itemize} 
Clearly, the only contraction satisfying these constraints is given by
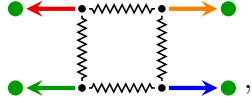
\begin{equ}
\begin{tikzpicture}[scale=0.35,baseline=0cm]
	\node at (0,0)  [dot] (bl) {};
	\node at (3,0)  [dot] (br) {};
	\node at (3,3)  [dot] (tr) {};
	\node at (0,3)  [dot] (tl) {};
	\node at (-2.5,0)  [root] (rbl) {};
	\node at (-2.5,3)  [root] (rtl) {};
	\node at (5.5,0)  [root] (rbr) {};
	\node at (5.5,3)  [root] (rtr) {};

	\draw[testfcn] (bl) to  (rbl);	
	\draw[testfcn,darkblue] (br) to  (rbr);	
	\draw[testfcn,darkred] (tl) to  (rtl);	
	\draw[testfcn,orange] (tr) to  (rtr);	

	\draw[kernelBig] (bl) to (br);
	\draw[kernelBig] (br) to (tr);
	\draw[kernelBig] (tr) to (tl);
	\draw[kernelBig] (tl) to (bl);
\end{tikzpicture}\;,
\end{equ}
modulo permutations of the four test functions.
The fact that such a term converges to $0$ as $\eps \to 0$ is then a consequence of 
Weinstein's theorem \cite[Prop.~2.3]{Hai17}, noting that Lemma~\ref{lem:boundKeps} 
implies that, for $\kappa \in [0,3/2)$, the kernel $K_\eps$ is of degree 
$-3/2-\kappa$ and the corresponding norm is of order $\eps^\kappa$.
\end{proof}

\subsection{Regularity structure and renormalisation}

We now describe the regularity structure associated to \eqref{e:basich}
by the procedure of \cite{BHZ} (see also \cite[Sec.~15.2]{Book}) and we argue 
that the BPHZ lift of our noise only leads to constant counterterms (as opposed to additional counterterms of the type $C_\eps \d_x h$ which is the only
other kind of counterterm that can possibly appear as a consequence of the general result
of \cite{Renorm}).

In the language of \cite{BHZ}, the regularity structure $(\CT,\CG)$ in question is generated by
one integration operator $\CI$ of regularity $2$, two noises $\Xi$ (corresponding to $\xi_\eps$) and $H$
(corresponding to $\chi_\eps$) of 
respective regularities $-3/2-\kappa$ and $-3/4-\kappa$ for some positive 
$\kappa>0$, and the normal complete rule containing the nodes $H\CI'$ and $(\CI')^2$.
It is straightforward to verify that this rule is subcritical and also satisfies
\cite[Ass.~2.31]{SGap} (at least for $\kappa$ small enough and we henceforth fix such a value).
For the sake of intuition it helps to represent the canonical basis vectors of $\CT$
by trees with possible nodes given by
\begin{equ}[e:nodes]
\<N3> \simeq \Xi\;,\quad
\<N4> \simeq H\;,\quad
\<N5> \simeq H\CI'\;,\quad
\<N1> \simeq \CI'\;,\quad
\<N2> \simeq (\CI')^2\;.
\end{equ}
We will use the convention that $\CI'$ annihilates Taylor monomials, so that we 
allow no node of type $1 \simeq \<one>$ except for the symbol $1$ itself.
For example, we have
\begin{equ}[e:example]
\CI'(H\CI'(\Xi))\CI'(H) = \<ex>\;,
\end{equ}
which contains nodes of all the above types except the penultimate one and is of 
degree $3+2(-3/4-\kappa)-3/2-\kappa = -3\kappa$. 
The prescription of \cite{BHZ} also requires us
to include basis vectors in $\CT$ corresponding to the classical Taylor monomials $X^k$, 
as well as trees with additional monomials attached at their nodes. The latter however
can be completely disregarded if we restrict ourselves to solutions of a low enough
regularity $\gamma$ for which \eqref{e:basich} is well-posed in a space of modelled
distributions that are locally of class $\CD^\gamma$ ($\gamma = 3/2 + 2\kappa$ will do
provided that $\kappa$ is small enough).

To see what degrees can possibly appear, note that each node in \eqref{e:nodes}
has one incoming edge (at the bottom) and between zero and two outgoing edges.
Furthermore, we should think of each basis vector of $\CT$ as represented by a tree
which has one incoming edge (at the root) and no outgoing edge. As a consequence, each
symbol (except $1$) must contain at least one node of type \<N3> or \<N4> and, for each
additional such node, one node of type \<N2>. As a consequence, the possible degrees
that can be realised are 
\begin{equ}
N_{\<N1>}
+ 2N_{\<N2>}
- N_{\<xi>} \bigl(\f32+\kappa\bigr)
- N_{\<H>} \bigl(\f34+\kappa\bigr)
+ N_{\<N5>}\bigl(\f14-\kappa\bigr)
\;,\quad N_\tau \in \N\;,\quad N_{\<xi>} + N_{\<H>} = 1+N_{\<N2>}\;.
\end{equ}
Note that, using the second relation, we find that the degree of a tree $\tau$
is given by
\begin{equ}[e:degtau]
\deg\tau  = N_{\<N1>}
+ N_{\<xi>} \bigl(\f12-\kappa\bigr)
+ N_{\<H>} \bigl(\f54-\kappa\bigr)
+ N_{\<N5>}\bigl(\f14-\kappa\bigr)
-2\;.
\end{equ}
In the case of \eqref{e:example} for example one has $N_{\<xi>} = N_{\<H>} = N_{\<N5>} = N_{\<N2>} = 1$
and $N_{\<N1>}=0$.

Fix now once and for all a truncation $\hat K$ of the heat kernel, i.e.\ $\hat K \colon \R^2 \to \R$
is such that $\hat K = P$ in a ball of radius $\f12$ around the origin, $\hat K = 0$ outside
of a ball of radius $1$, and $\int \hat K(z)\,dz = 0$. We furthermore assume that $\hat K$
is symmetric under spatial reflections, namely that 
\begin{equ}
\hat K(t,-x) = \hat K(t,x)\;.
\end{equ}
We then write $Z^{(\eps)} = (\Pi^{(\eps)},\Gamma^{(\eps)})$ for the canonical lift of the noise 
$(\eta_\eps,\xi_\eps)$ into the regularity structure $\CT$. Recall that this is built
by first inductively defining a linear map $\PPi^{(\eps)}\colon \CT \to \CC^\infty$ by setting
$\PPi^{(\eps)} 1 = 1$, $\PPi^{(\eps)} \<xi> = \xi_\eps$, $\PPi^{(\eps)} \<H> = \eta_\eps$, and then
$\PPi^{(\eps)} \CI'(\tau) = \d_x \hat K \star \PPi^{(\eps)} \tau$,
$\PPi^{(\eps)} \bar\tau\tau = \PPi^{(\eps)} \bar \tau \cdot\PPi^{(\eps)} \tau$.
One then defines $Z^{(\eps)}$ by a suitable ``recentering'' procedure as in \cite[Def.~6.8]{BHZ}.

Note that each $\PPi^{(\eps)}\tau$ is a stationary process and we set
\begin{equ}[e:defCeps]
C^{(\eps)}_\tau = \E \PPi^{(\eps)}\tau(0)\;,
\end{equ}
see \cite[Eq.~6.24]{BHZ}. We then have the following elementary result.

\begin{proposition}\label{prop:zero}
If $\tau$ is such that either $N_H^\tau \eqdef N_{\<H>}+N_{\<N5>}$ is odd or 
$N_{\CI'}^\tau \eqdef N_{\<N1>} + 2N_{\<N2>}+N_{\<N5>}$ is odd, then
$C^{(\eps)}_\tau = 0$.
\end{proposition}

\begin{proof}
Since $\xi_\eps$ is an element of the \textit{second} (homogeneous) Wiener chaos while
$\eta_\eps$ is an element of the first chaos, it follows that,
whenever  $N_{\<H>}$ is odd, $\Pi^{(\eps)}\tau$ is a random variable
belong to a sum of odd chaoses, so that its expectation vanishes.

If $N_{\<N1>}$ is odd then $C^{(\eps)}_\tau$ is given by the expectation of an integral
involving an odd number of instances of the kernel $\d_x \hat K$. Since the latter is odd
under spatial reflections but the covariances of $\xi_\eps$ and $\eta_\eps$ are even,
it follows from reflecting all integration variables that $C^{(\eps)}_\tau=-C^{(\eps)}_\tau$
as claimed.
\end{proof}

Note now that the actual renormalisation constants $\hat C^{(\eps)}_\tau$ appearing
in the BPHZ lift of the model (and therefore also multiplying the counterterms
associated to each tree $\tau$ of negative degree) are obtained by setting 
$\hat C^{(\eps)}_\tau = C^{(\eps)} \tilde \CA_- \tau$, where $C^{(\eps)}$ was
defined in \eqref{e:defCeps} and $\tilde \CA_-$ is the 
``negative twisted antipode'' given in \cite[Prop.~6.6]{BHZ}.
We only need to know that these are given by polynomial expressions in the 
$C^{(\eps)}_\tau$ that have the property that every monomial 
of the form $\prod C^{(\eps)}_{\tau_i}$ appearing in the definition of $\hat C^{(\eps)}_\tau$
satisfies $N_H^\tau = \sum_i N_H^{\tau_i}$ and $N_{\CI'}^\tau = \sum_i N_{\CI'}^{\tau_i}$
(and similarly for $N_{\<xi>}$).
This is an immediate consequence of the definition \cite[Eq.~6.25]{BHZ} 
and the defining properties of the ``twisted
antipode'' given in Equations~6.8--6.9 of that article. 
Since any sum of integers adding to an odd number must contain
at least one odd summand and since odd moments of Gaussian random variables vanish, this 
implies the following.

\begin{corollary}\label{cor:counterterm}
Under the assumption of Proposition~\ref{prop:zero} one also has $\hat C^{(\eps)}_\tau = 0$.\qed
\end{corollary}

Recall now that, when considering \eqref{e:basich} driven by the BPHZ lift of
$(\chi_\eps,\xi_\eps)$, each tree $\tau$ with $\deg\tau < 0$ leads to 
an additional counterterm in the right-hand side of the form 
$\hat C^{(\eps)}_{\tau} F_\tau(\d_x h)$. We now claim that, for every such $\tau$, one
has either $F_\tau(u) \propto u$ or $F_\tau(u) \propto 1$ and, furthermore, that 
one has $\hat C^{(\eps)}_{\tau}=0$ for those
$\tau$'s such that $F_\tau(u) \propto u$.

Indeed, one can see from \cite{Renorm} (see Equation~2.12 there for the definition of 
the counterterms $\Upsilon^F[\tau]$, as well as Equation~3.9 and Theorem~3.25 for a 
description of 
how this relates to the renormalised equation) 
that, given $\tau$, every node of $\tau$ of type
\<N4> or \<N1> leads to a factor $\d_xh$ while every other node leads to a constant factor,
so that $F_\tau(u)\propto u^{N_{\<H>}+N_{\<N1>}}$. If $\tau$ is a tree such that 
$N_{\<H>}+N_{\<N1>} \ge 2$,
then it follows from \eqref{e:degtau} combined with the fact that one has the constraint
$N_{\<xi>} + N_{\<H>} \ge 1$ that $\deg \tau \ge \f94-\kappa > 0$, so that counterterms 
of the type $(\d_xh)^2$ (and higher) are indeed ruled out.
The main claim of this section is the following.

\begin{proposition}
The BPHZ counterterm of \eqref{e:basich} is given by $\sum_{\tau\,:\, \deg \tau < 0} \hat C^{(\eps)}_\tau$.
\end{proposition}

\begin{proof}
We only need to show that $\hat C^{(\eps)}_\tau = 0$ whenever $\tau$ is such that  
$N_{\<H>}+N_{\<N1>} = 1$ and $\deg \tau < 0$. Since we want to apply the criterion 
from Corollary~\ref{cor:counterterm}, it is enough to show that if $\tau$ is such that its conditions fail 
(i.e.\ $N_{\<H>}+N_{\<N5>}$ and $N_{\<N1>} + N_{\<N5>}$ are both even) and 
$N_{\<H>}+N_{\<N1>} = 1$ then $\deg\tau > 0$.

Let us first consider the case $N_{\<N1>} = 1$. This then implies that $N_{\<N5>}$ is odd
(and in particular at least $1$) so that $N_{\<H>}$ is also odd, which leads to a contradiction
with the constraint $N_{\<H>}+N_{\<N1>} = 1$.
If $N_{\<H>} = 1$, then the same argument again leads to a contradiction, thus completing the proof.
\end{proof}

\subsection{Convergence of the model}
\label{sec:convergenceKPZ}

In view of the results of the previous section and the results of 
\cite{Hai14,BHZ,Renorm}, it remains to show 
that, in the context of a suitable regularity structure associated to \eqref{e:basich}, 
the BPHZ lift of the noises $(\chi_\eps,\xi_\eps)$ converges to the BPHZ lift of 
$(0,\xi)$ with $\xi = \lim_{\eps \to 0}\xi_\eps$ a space-time white noise.

As in \cite{SGap}, we now introduce a family of modelled distributions that will describe 
the derivative of our model with respect to the noise and that will be
used to inductively bound the BPHZ lift of $(\chi_\eps,\xi_\eps)$. The only difference is
that the pair $(\chi_\eps,\xi_\eps)$ by itself doesn't appear to satisfy a spectral 
gap inequality. The underlying Gaussian noise $\eta$ however does of course, so this is 
what we will be exploiting. As a consequence, our setup is slightly different
from that of \cite{SGap}. In particular, we will work with $L^p$-based spaces for $p$ slightly 
smaller than $2$ instead of the $L^2$-based spaces used in \cite[Secs~4 \& 5]{SGap}. Since the 
analytic core of that article (Section~3.2) allows for arbitrary $p$, this does not
cause any difficulty.

Let now $\gamma_\tau$ and $\deg_2\tau = \deg \tau + \f32$ be defined
as in \cite[Sec.~4]{SGap}, and fix some $p<2$ (we will have to choose it very close 
to $2$ later on, for example $p = 2-10^{-3}$ will do).
Here, $\gamma_\tau$ represents the regularity (as opposed to homogeneity) of the 
function \slash distribution $\Pi_x\tau$.
We also recall from \cite[Sec.~4]{SGap} that one can introduce spaces of modelled
distributions $\CD_p^\gamma$ which are the natural $p$-Besov analogues of the usual 
spaces $\CD^\gamma$ (which would correspond to $p=\infty$). The spaces
that play the star role in that article are spaces $\CD_p^{\gamma,\nu;z}$ of ``pointed''
modelled distributions which one should think of as elements of $\CD_p^\gamma$ that 
furthermore vanish at order $\nu$ (in an $L^p$-integrated sense) around the point $z$.
The norms of these spaces are chosen precisely in such a way that they are weak enough to support
an inductive construction of (pointed) modelled distributions $H^{z,h}_{\tau;\eps}$ representing 
the variation
of $\Pi_x \tau$ under an infinitesimal shift of the underlying white noise $\eta$ in a 
Cameron--Martin direction $h$, but strong enough to yield near-optimal bounds on their reconstruction
that are compatible with the usual bounds on models. Note here that the quantity $\f32$ by which
$\deg_2\tau$ differs from $\deg\tau$ is precisely the amount of regularity that one
loses by Sobolev embedding when going back from $L^2$-based bounds to $L^\infty$-based bounds.

Given any $h \in L^2$, we then define 
a collection of pointed modelled distributions $H^{z,h}_{\tau;\eps} \in \CD_p^{\gamma_\tau,\deg_2\tau;z}$
for the canonical basis vectors
$\tau$ of $\CT$ inductively by setting
\begin{equ}[e:defStartInduction]
H^{z,h}_{X^k;\eps} = 0\;,\quad H^{z,h}_{H;\eps} = \eps^{3/4} (\d_x^2P \star \rho_\eps \star h) \,\one\;,\quad
H^{z,h}_{\Xi;\eps} = 0\;,
\end{equ}
as well as
\begin{equ}[e:induction]
H^{z,h}_{\tau\bar\tau;\eps} = H^{z,h}_{\tau;\eps}f^{z}_{\bar\tau} + f^{z}_{\tau}H^{z,h}_{\bar\tau;\eps}\;,\quad
H^{z,h}_{\CI\tau;\eps} = \CK^{x,p}_{\gamma_\tau,\deg_2\tau}H^{z,h}_{\tau;\eps}\;.
\end{equ}
(See \cite[Sec.~4]{SGap} for more details of a very similar construction.)
As in \cite{SGap}, the second identity of \eqref{e:induction} only makes sense a priori
for $\gamma_\tau > 0$. This fails precisely when $\tau = \Xi$ (and only then). 

In that case, we \textit{postulate} that the reconstruction of 
$H^{z,h}_{\Xi;\eps}$ is given by
\begin{equ}[e:defReconstr]
\CR H^{z,h}_{\Xi;\eps}
= 2\eps^{3/4}\chi_\eps (\d_x^2P \star \rho_\eps\star h)\;,
\end{equ}
and we define $H^{z,h}_{\CI\Xi;\eps}$ using \cite[Thm~3.21]{SGap}.
The following analogue to \cite[Lems~4.3 \& 4.5]{SGap} provides the starting point for our induction.

\begin{lemma}\label{lem:bound}
For any $\alpha < 3/4$ and $p \le 2$, we have $H^{z,h}_{H;\eps} \in \CD_p^{\alpha,\alpha;z}$
uniformly over $h \in L^2$ and there exist $\kappa > 0$ and $C > 0$ (depending on $\alpha$) 
such that
\begin{equ}
\sup_{h \in L^2} \$H^{z,h}_{H;\eps}\$_{p,\alpha,\alpha;z} \le C\eps^\kappa\;,\qquad \forall\eps \in (0,1]\;.
\end{equ}
Furthermore, the right-hand side of \eqref{e:defReconstr} is almost surely a candidate for the 
pointed reconstruction in the sense of
\cite[Def.~3.20]{SGap} of $\smash{H^{z,h}_{\Xi;\eps} = 0}$ viewed as an element of 
$\smash{\CD_p^{\gamma_\tau,\deg_2\tau;z}}$. Recalling the notation $C(f;B_z)$ appearing in that 
definition which provides a ``norm'' for such a pair $(f,\CR f)$, one has
\begin{equ}
\sup_{\eps \le 1} \E \Big(\sup_{\|h\| = 1}\sup_{z \in K} C(\smash{H^{z,h}_{\Xi;\eps}};B_z)\Big)^q < \infty\;,
\end{equ}
for all $q \ge 1$.
\end{lemma}

\begin{proof}
Regarding the first claim, it follows immediately from the definition of local Besov
spaces (in $1+1$-dimensional parabolic space-time) \cite[Sec.~1.3]{SGap} that maps of the type
$f \mapsto \eps^\alpha \rho_\eps \star f$ (with $\alpha \ge 0$) are 
uniformly bounded over $\eps \in (0,1]$ from $L^p$ into the local
Besov space $\CB_p^{\alpha}$.

Furthermore, as a consequence of \cite[Prop.~A.5]{BL21}, 
$\CB_p^{\alpha}$ embeds canonically into $\CD_p^{\alpha}$ and, provided 
that $\alpha < 3/p$, the space $\CD_p^{\alpha}$ (with values in the polynomial sector)
coincides with $\CD_p^{\alpha,\alpha;z}$. This latter statement can be shown in 
exactly the same way as \cite[Lem.~4.3]{SGap}.

The second claim follows in an analogous way once we note that 
$\E \chi_\eps^2(z) \approx \eps^{-3/2}$ so that $\eps^{3/4}\chi_\eps$ belongs to every
local $L^p$ space and admits moments of all orders. 
\end{proof}

At the technical level, the main result of this section is the following.

\begin{theorem}\label{theo:convLift}
Let $\hat Z^{(\eps)}$ denote the BPHZ lift of $(\chi_\eps,\xi_\eps)$. 
Then, one has the convergence in law $\hat Z^{(\eps)} \to Z$, where $Z$
denotes the BPHZ lift of  $(0,\xi)$ with $\xi$ as in Proposition~\ref{prop:convNoise}.
\end{theorem}

\begin{proof}
We first show that $\hat Z^{(\eps)}$ is bounded uniformly over $\eps \in (0,1]$.
The proof by induction is virtually identical to that of \cite[Thm~2.33]{SGap}, the only difference being the starting point for our induction. 
The bounds of Lemma~\ref{lem:bound} provide the starting point for an inductive proof that,
for every $\tau \in \CT$ there exists $k \ge 1$ and, for every $z \in \R^2$, there
exists a compact set $\K$ containing $z$ such that
\begin{equ}[e:boundUnifH]
\sup_{\|h\| \le 1}\sup_{\eps \le 1}
 \$H^{z,h}_{\tau;\eps}\$_{p,\gamma_\tau,\deg_2\tau;z} 
 \lesssim (1+\|\hat Z^{(\eps)}\|_{\CT_\tau;\K})^k\;.
\end{equ}
Here, $\gamma_\tau$ and $\deg_2\tau$ are defined as in \cite[Sec.~4]{SGap} and
$\CT_\tau$ denotes the smallest sector of $\CT$ such that $\CT_\tau + \Vec \tau$
is also a sector.
The proof of this statement is identical to that of \cite[Prop.~4.7]{SGap}.
In fact, it is easier since $\tau = \Xi$ is the only element with $\gamma_\tau \le 0$,
so that we never need to use \cite[Lem.~4.6]{SGap} which was one of the more delicate points
in that proof.

The definitions
\eqref{e:defStartInduction} and \eqref{e:defReconstr} guarantee that one has 
the identities
\begin{equ}
D_h \Pi_z^{(\eps)}\Xi = \CR H_{\Xi;\eps}^{z,h}\;,\qquad
D_h \Pi_z^{(\eps)}H = \CR H_{H;\eps}^{z,h}\;,
\end{equ}
where $D$ denotes the Fréchet derivative with respect to the underlying space-time
white noise $\eta$ in the direction $h \in L^2$. This again provides the starting point for 
the inductive proof that the identity $D_h \Pi_z^{(\eps)}\tau = \CR H_{\tau;\eps}^{z,h}$
holds for every $\tau \in \CT$, the proof being identical to that of 
\cite[Thm~4.18]{SGap}. Combining these two facts with the Gaussian spectral gap 
inequality then yields the uniform bound on  $\hat Z^{(\eps)}$ exactly as in
\cite[Thm~2.33]{SGap}.

Regarding the convergence as $\eps \to 0$, we again argue in a way similar 
to \cite{SGap}, but with significant simplifications and one additional 
layer of approximation. For $\delta > 0$, we define
\begin{equ}
\xi_{\eps,\delta} = \rho_\delta \star \xi_\eps\;,\qquad
\chi_{\eps,\delta} = \rho_\delta \star \chi_\eps\;,
\end{equ}
and we denote by $\hat Z^{(\eps,\delta)}$ the BPHZ lift of these processes.
As a consequence of Proposition~\ref{prop:convNoise}, of \cite[Thm~2.33]{SGap} (using the fact
that $\xi$ satisfies a spectral gap inequality in $L^2$)
and of the fact that
all these processes are smooth, we know that one has the limits in probability
(in the topology of convergence of models on compact sets introduced in \cite[Eq.~2.17]{Hai14})
\begin{equ}
\lim_{\delta \to 0}\lim_{\eps \to 0}\hat Z^{(\eps,\delta)} = Z\;,\qquad
\lim_{\delta \to 0}\hat Z^{(\eps,\delta)} = \hat Z^{(\eps)}\;,
\end{equ}
so our aim is to exchange the limits in the first bound. This is immediate if we can 
get a bound on the distance between $\hat Z^{(\eps,\delta)}$ and $\hat Z^{(\eps)}$
that goes to $0$ as $\delta\to 0$, uniformly over $\eps \in (0,1]$. The proof that this is 
the case is virtually identical to that of \cite[Thm~6.9]{SGap}, the starting point
of the argument being provided by the fact that if we set
\begin{equ}
H^{z,h}_{H;\eps,\delta} = \eps^{3/4} (\rho_\delta \star \d_x^2P \star \rho_\eps \star h) \,\one\;,
\end{equ}
as well as
\begin{equ}
H^{z,h}_{\Xi;\eps,\delta} = 0\;,\qquad
\CR H^{z,h}_{\Xi;\eps,\delta} = 2\eps^{3/4} \rho_\delta \star \big(\chi_\eps (\d_x^2P \star \rho_\eps\star h)\big)\;,
\end{equ}
then both $\sup_{h \in L^2} \$H^{z,h}_{H;\eps}-H^{z,h}_{H;\eps,\delta}\$_{p,\alpha,\alpha;z}$
and the norm of $\CR H^{z,h}_{\Xi;\eps}-\CR H^{z,h}_{\Xi;\eps,\delta}$ converge to $0$
as $\delta\to 0$,
uniformly over $\eps \in (0,1]$. One then defines $H^{z,h}_{\tau;\eps,\delta}$
in the same way as $H^{z,h}_{\tau;\eps}$, but with all manipulations done with respect
to the model given by the BPHZ lift of $(\chi_{\eps,\delta},\xi_{\eps,\delta})$.
By repeatedly applying \cite[Thms~3.11, 3.19, 3.21]{SGap}, it is then straightforward 
to obtain the bound
\begin{equ}
\$H^{z,h}_{\tau;\eps};H^{z,h}_{\tau;\eps,\delta}\$_{p,\gamma_\tau,\deg_2\tau;z}
\lesssim o(\delta) (1+\|\hat Z^{(\eps)}\|_{\CT_\tau;\K}+\|\hat Z^{(\eps)};\hat Z^{(\eps,\delta)}\|_{\CT_\tau;\K})^k\;,
\end{equ}
with quantifiers over $h$, $\tau$, $z$, etc as in \eqref{e:boundUnifH}, and $o$ some function
with $\lim_{\delta \to 0}o(\delta) = 0$.
The reason why these theorems are sufficient in our case is again that 
 $\gamma_\tau > 0$ for all $\tau \neq \Xi$, so we never have coefficients of negative
 regularity appearing in our modelled distributions. 
\end{proof}

We now have all the ingredients in place to provide a proof of the first main theorem of this
article.

\begin{proof}[Proof of Theorem~\ref{theo:main:KPZ}]
Recall the 
construction in \cite[Sec.~1.5.1]{YangMills} which, given an unbounded complete separable
metric space $X$ with a distinguished ``origin'' $o \in X$, constructs a space
$X^\sol$ which essentially consists of continuous functions $\R_+ \to \hat X$, where $\hat X$ is obtained by adjoining to $X$ a point at infinity, such that once the function blows up it cannot
be ``reborn''. The space $X^\sol$ is itself a complete separable metric space
and, if $f$ blows up at some time $\tau>0$, then any sequence of functions $g_\eps$
such that $d(f(t),g_\eps(t)) \le \eps$ for $t \le \tau - \eps$ converges to $f$ in $X^\sol$.
On the other hand, if $f \in \CC(\R_+,X)$, then $d(g_\eps, f) \to 0$ if and only if
the functions $g_\eps$ converge to $f$ uniformly over compact time intervals. 

The main result of \cite{Hai13,Hai14,Renorm} is that, for every $\alpha < 1/2-\kappa$, there 
exists a jointly continuous map 
\begin{equ}
\CS \colon \CC^\alpha \times \CM_0 \to (\CC^\alpha)^\sol\;,
\end{equ}
such that, writing as above $\hat Z^{(\eps)}$ for the BPHZ lift of $(\chi_\eps,\xi_\eps)$, 
$\CS(h_0,\hat Z^{(\eps)})$ coincides with the solution to
\begin{equ}[e:basichRenorm]
\d_t h = \d_x^2 h +  (\d_x h)^2 + 2\chi_\eps \d_x h + \xi_\eps - C_\eps  \;,
\end{equ}
with initial condition $h_0$, for a suitable constant $C_\eps$.
Furthermore, if $Z$ denotes the BPHZ lift of $(0,\xi)$, then 
$\CS(h_0,Z)$ is known by \cite[Rem.~1.8]{BGHZ22} to coincide with 
the Cole--Hopf solution to the KPZ equation, provided that one possibly changes the
value of $C_\eps$ by a fixed quantity of order one.

The claim then follows from Theorem~\ref{theo:convLift}, combined with the fact 
that solutions to the KPZ equation are global in time.
\end{proof}

\section{Fractional Brownian motion}

We now turn to the proof of convergence of SDEs driven by mollified 
fractional Brownian motions to a Markov process diffusing in the direction of
the Lie brackets of the vector fields defining the original equation.

%
%\begin{lemma}
%Let $f_i \in L^1_\loc(\R)$, $\delta_i > 0$, and $a_i \in \R$ with $a_1 + a_2 \neq 1$ 
%be such that $|f_i(t)| \le \delta_i |t|^{-a_i}$, 
%let $T \in (0,1]$, let $[s,t] \cup [u,v] \subset [0,T]$, and write
%\begin{equ}
%F(s,t) = \int_u^v f_1(t-r)f_2(r-s)\,dr\;.
%\end{equ}
%Then, if $a_i < 1$, there exists a constant $C$ such that
%\begin{equ}
%|F(s,t)| \le C \delta_1\delta_2 \big(|t-s|^{1-a_1-a_2} + T^{1-a_1-a_2}\big)\;.
%\end{equ}
%If $a_1 = 1$ and $a_2 < 1$, then 
%\begin{equ}
%|F(s,t)| \le C \delta_2 \big((\|f_1\|+\delta_1)|t-s|^{-a_2} + \delta_1 T^{-a_2}\big)\;.
%\end{equ}
%where $\|f\| = \int_{-1}^1|f(r)|\,dr$.
%\end{lemma}
%
%\begin{proof}
%Writing $|F(s,t)| \le \int_0^T |f_1(t-r)f_2(r-s)|\,dr$, we break the integral into three
%regions: $|r-t| \le |t-s|/2$, $|r-s| \le |t-s|/2$, and $|r-(s+t/2)| \ge |t-s|$.
%The integral $I_1$ over the first region is bounded by
%\begin{equ}
%I_1 \lesssim \delta_2 |t-s|^{-a_2} \int_{-|t-s|}^{|t-s|} |f_1(r/2)|\,dr\;,
%\end{equ}
%and similarly for $I_2$ with the indices $1$ and $2$ exchanged. The remainder $I_3$ is itself bounded by
%\begin{equ}
%I_3 \lesssim \delta_1\delta_2 \int_{|t-s|}^{2T} r^{-a_1-a_2}\,dr
%\lesssim \delta_1\delta_2 \bigl(|t-s|^{1-a_1-a_2} + T^{1-a_1-a_2}\bigr)\;,
%\end{equ}
%provided that $a_1 + a_2 \neq 1$. It is then straightforward to combine these 
%bounds in order to deduce the claim.
%\end{proof}

Let $\rho$ be a fixed mollifier (smooth, with support in $[-1,1]$, integrating to $1$),
write $\rho_\eps(t) = \eps^{-1}\rho(t/\eps)$, and set
$B_\eps = C_\eps B \star \rho_\eps$ where $B$ denotes a two-sided $m$-dimensional fractional Brownian motion 
with Hurst parameter $H \in (0,\f14]$ and $C_\eps$ is as in Theorem~\ref{thm:mainfbm}.
We furthermore write $\xi_\eps = \dot B_\eps$ and we lift it to a geometric rough path
in the usual way \cite{Book} by setting
\begin{equ}[e:defLift]
\X^{(k)}_\eps (s,t) = \int_s^t \int_s^{r_k}\cdots \int_s^{r_2} \xi_\eps(r_1)\otimes \ldots \otimes \xi_\eps(r_k)\,dr_1\ldots dr_k\;.
\end{equ}
Our proof relies on showing that $\X^{(k)}_\eps$ converges to $0$ for $k$ odd, while
$\X^{(2n)}_\eps$ converges to the $n$-fold iterated integral of an $m^2$-dimensional 
Wiener process $W$ satisfying the antisymmetry constraint $W_{ij} = -W_{ji}$.
Since we work in the space of $\alpha$-Hölder rough paths for some $\alpha \in (\f15,\f14)$,
it is in fact sufficient to show that the first $4$ iterated integrals converge in the corresponding topology,
since convergence of iterated integrals of higher order then follows at once.

Even though the situation is quite similar to that of the KPZ equation treated in the previous 
section, it turns out that the proof is significantly more involved. This is mainly
a consequence of the fact that, in the language of regularity structures, the 
symbols $\Xi_i \CI(\Xi_j)$ which will end up representing the (time derivative of the) 
Wiener processes $W_{ij}$ in the limit transform non-trivially under the action of the structure group.
This is unlike the case of the KPZ equation where the second-order process $\xi_\eps$ 
can simply be represented by a single noise symbol which transforms trivially.

Our proof relies on delicate estimates on mixed moments of the quantities $\X^{(k)}_\eps (s,t)$.
Since the integration variables $r_i$ appearing in the expression \eqref{e:defLift} are ordered
like $s \le r_1 \le\ldots\le r_k \le t$, partially ordered sets (posets) will be
a convenient structure used to index our integration variables.

\subsection{Integrations indexed by posets}

In order to express and estimate moments of \eqref{e:defLift}, the following formalism will prove to be useful.
Given a finite poset $P$, denote $P^\top$ its maximal elements, $P^\bot$ its minimal elements,
and $P^\circ = P\setminus (P^\top \cup P^\bot)$.
Given functions $s\colon P^\bot \to \R$ and $t \colon P^\top \to \R$, we then write
$[s,t]_P$ for the set of all monotone functions $r \colon P \to \R$ such that 
$r\restr P^\bot = s$ and $r\restr P^\top = t$. 
%(When convenient, we also identify
%a real value $s \in \R$ with the corresponding constant function $s\colon P^\bot \to \R$
%without further mention, and similarly for $t$.)
 We say that
$P$ is ``linear'' if every $x \in P^\circ$
has a unique immediate predecessor $x^\down$ and a unique immediate successor $x^\up$.
Note that, given linear posets $P_1$ and $P_2$, their disjoint sum $P_1 \sqcup P_2$ 
is again linear.

Given a finite set $V$, we write $\CP(V)$
for the set of pairings of $V$, namely the set of partitions $E$ of $V$ such that every
element of $E$ contains exactly two elements of $V$. (Note that for $\CP(V)$ to be
non-empty $V$ needs to have an even number of elements.)
Given a map $r \colon V \to \R$ and a pairing $G \in \CP(V)$, we then set
\begin{equ}[e:defKepsG]
K_\eps^{(V,G)}(r) = \prod_{\{e,f\} \in G} K_\eps(r_f-r_e)\;,
\end{equ}
where we set $K_\eps(t) = \E \xi_\eps(0)\xi_\eps(t)$ and note that this
expression is well-defined since $K_\eps$ is even.

With these notations in place, we can write down an expression for the joint moments
of the components of $\X^{(k)}_\eps$. This is already hinted at by noting that the 
integral in \eqref{e:defLift} runs over $[s,t]_{P[k]}$, where $P[k] = \{0,\ldots,k+1\}$
endowed with its natural order. Indeed, fix $n \ge 1$, lengths $k_1,\ldots,k_n$,
 multiindices $I_i \in \{1,\ldots,m\}^{k_i}$, intervals $[s_i,t_i]$, and fix the shorthand
 $\X_{\eps}[i] = \scal{\X_{\eps}^{(k_i)}(s_i,t_i),e_{I_i}}$ with $e_{I_i}$
 the corresponding basis vector of $(\R^m)^{\otimes k_i}$. We then set
 $P = P[k_1]\sqcup\dots\sqcup P[k_n]$ and identify $s$ and $t$ with the corresponding
constant functions on $P^\bot$ and $P^\top$ respectively. By Wick's formula, there then exists a
 subset $\CP^I \subset \CP(P^\circ)$ of the pairings of $P^\circ$ such that 
\begin{equ}[e:moments]
\E\Big( \prod_{i=1}^n \X_{\eps}[i]\Big) = \sum_{G \in \CP^I} \int_{[s,t]_P} K_\eps^{(P^\circ,G)}(r\restr P^\circ)\,dr\;.
\end{equ}
When trying to estimate such expressions, one annoying feature is that 
$\int \abs{K_\eps(t)}\,dt$ diverges as $\eps \to 0$. As a consequence, we cannot 
simply replace the integrand by its absolute value but have to exploit cancellations. 
For this, our main tool is a systematic way of rewriting the right-hand
side of \eqref{e:moments} in terms of the kernel $\bar K_\eps$ given by $\bar K_\eps(t) = \int_0^t K_\eps(r)\,dr$.

Given a finite set $V$ with a distinguished subset $V^\circ \subset V$, 
write $\CG_0(V)$ for the set of all pairs $(G,E)$ such that
 $G$ is a collection of 
$2$-element subsets of $V^\circ$ (undirected edges) and $E \subset V\times V$ (directed edges).
We write $\CG(V) \subset \CG_0(V)$ for those pairs $(G,E)$ such that
 their local structure around any vertex $v \in V^\circ$ is furthermore of one
of the following three types:
\begin{enumerate}
\item There is exactly one $g \in G$ with $v \in g$ and no edge of $E$ touching $v$.
\item There is exactly one edge $e \in E$ pointing out of $v$, i.e.\ of the 
type $(v,v')$ for some $v' \in V$, and no edge in $G$ containing $v$.
\item There is exactly one edge $e \in E$ pointing into $v$, no edge in $E$ pointing
out of $v$, and exactly one edge in $G$ containing $v$.
\end{enumerate}
Furthermore, we impose that there is at most one vertex of type~3
and that vertices in $V \setminus V^\circ$ only have edges in $E$ pointing into them.
If we draw undirected edges in $G$ as dark blue lines and directed edges in $E$ as light blue arrows,
examples of vertices of these types are as follows:
\begin{equ}
\text{type 1: }
\begin{tikzpicture}[baseline=-.08cm]
\node[dot] (1) at (0,0) {};

\draw[very thick,darkblue] (1) to ++(0.8,0);
\end{tikzpicture}
\;,\quad\text{type 2: }
\begin{tikzpicture}[baseline=-.08cm]
\node[dot] (1) at (0,0) {};

\draw[thick,lightblue,<-] (1) to ++(-0.7,0.3);
\draw[thick,lightblue,<-] (1) to ++(-0.7,-0.3);
\draw[thick,lightblue,->] (1) to ++(0.8,0);
\end{tikzpicture}
\;,\quad\text{type 3: }
\begin{tikzpicture}[baseline=-.08cm]
\node[dot] (1) at (0,0) {};

\draw[thick,lightblue,<-] (1) to ++(-0.8,0);
\draw[very thick,darkblue] (1) to ++(0.8,0);
\end{tikzpicture}\;.
\end{equ}
The following is an immediate consequence of this structure.

\begin{lemma}\label{lem:components}
Let $(G,E) \in \CG(V)$.
Each connected component of the directed graph $(V,E)$ is of one of the following three types:
\begin{enumerate}[label=(\alph*)]
\item\label{tpe:tree} A tree directed towards 
its root with the root belonging to $V \setminus V^\circ$ and all other vertices belonging to $V^\circ$.
\item\label{tpe:cycle} The union of a directed cycle $C$ 
of at least two vertices belonging to $V^\circ$, together with finitely many trees directed 
towards their roots belonging to $C$.
\item\label{tpe:line} A tree in $V^\circ$ directed towards 
its root which is the unique vertex of type~3 above.
\end{enumerate}
\end{lemma}

\begin{proof}
Since every vertex can have at most one outgoing edge in $E$ we can explore the graph $(V,E)$
by following these edges. Since $V$ is finite, we must eventually reach a directed cycle or a
vertex without outgoing edge. The latter must be either of type~2 or of type~3, thus accounting
for the three types of connected components. Note that a component cannot contain more than
one such limit set since their basins of attraction must intersect (by connectedness) which 
would force such an intersection point to have at least two outgoing edges which is not allowed.
\end{proof}

We then set, similarly to \eqref{e:defKepsG},
\begin{equ}
K_\eps^{(V,G,E)}(r) = \prod_{\{e,f\} \in G} K_\eps(r_f-r_e)\prod_{(e,f) \in E} \bar K_\eps(r_f-r_e)\;.
\end{equ}
Note that we require edges in $E$ to be directed since $\bar K_\eps$ is odd.
Consider now the case when the vertex set $V$ is given by a linear poset $P$ as above.
It will then be convenient to consider the $\Z$-module $\CG_0^\Z(P)$ generated
by $\bigcup_{\hat P \subset P} \CG_0(\hat P)$ where the union runs over 
sub-posets such that $\hat P^\top = P^\top$
and $\hat P^\bot = P^\bot$.
We define $\CG^\Z(P)$ analogously, except that we also quotient out the submodule spanned by graphs 
$(\hat P,G,E) \in \CG(P)$ such that $E$ contains a self-edge (i.e.\ an edge of the type $e = (v,v)$).

Given $s,t$ as above, we define the $\Z$-linear 
valuation $J_\eps(s,t) \colon \CG^\Z(P) \to \R$ such that 
\begin{equ}[e:defJeps]
J_\eps(s,t)(\hat P, G,E) = \int_{[s,t]_{\hat P}} K_\eps^{(\hat P,G,E)}(r)\,dr\;.
\end{equ}
This is well-defined since $K_\eps^{(\hat P,G,E)}$ vanishes if $E$ contains a self-edge as a
consequence of the fact that $\bar K_\eps(0) = 0$.
It will also be convenient to fix a total order on $P$. (In order to avoid confusion,
 choose it compatible with the given partial order. Our construction will depend on this
 total order, but the resulting bounds will not. Note also that the notations $v^\uparrow$ and $v^\downarrow$ introduced
 at the start of this section are derived from the partial order on $P$ and have no relation with the 
 total order we fixed for convenience.)
 
We then define a linear map $\CI \colon \CG^\Z(P) \to \CG_0^\Z(P)$\label{p:defI} in the following way.
If $G$ is empty, then $\CI(\hat P,G,E) = (\hat P,G,E)$. If not, we define $v^*$ as the
unique vertex of type~3 if such a vertex exists
and as the smallest (with respect to the fixed total order) vertex of type~1 otherwise. 
Either way, write $v_* \in \hat P^\circ$ for the unique 
element such that $g_* = \{v^*,v_*\} \in G$ which is then necessarily of type~1. 
We then set $\hat P' = \hat P \setminus \{v_*\}$, $G' = G \setminus \{g_*\}$, 
$E^\up = E \cup \{(v^*, v_*^\up)\}$, $E^\down = E \cup \{(v^*, v_*^\down)\}$,
as well as $g^\up = (\hat P',G',E^\up)$, and similarly for $g^\down$. With these
definitions, $\CI$ is given by
\begin{equ}[e:defCI]
\CI(\hat P,G,E) = g^\up - g^\down\;.
\end{equ}
This is indeed well-defined in the sense that $E^\up \subset \hat P'\times \hat P'$
and $G'$ consists of disjoint undirected edges of $(\hat P')^\circ$. 
In fact, $\CI$ is simply an encoding of the integration by parts formula as follows.

\begin{proposition}\label{prop:Integration}
The map $\CI$ maps $\CG^\Z(P)$ to $\CG^\Z(P)$ and one has the identity
\begin{equ}[e:integration]
J_\eps(s,t)(\CI g) = J_\eps(s,t)(g)\;,
\end{equ}
for every $g \in \CG^\Z(P)$. Furthermore, there exists
$N > 0$ such that $\CI^N = \CI^{N+1}$, which we then call $\CI^\infty$.
\end{proposition}

\begin{proof}
We first argue that one has indeed $(G',E^\up) \in \CG(\hat P')$ (the proof for $E^\down$
being identical). Whether $v^*$ is of type~3 or~1, $v_*$ is necessarily of type~1 since
$v_* \neq v^*$ and, if there exists any vertex of type~3 then this is $v^*$.
Furthermore the only edge touching $v^*$ is $g_*$ so that, after removing it
we are indeed only left with edges in $\hat P'$.

The addition of the edge $(v^*, v_*^\up)$ then guarantees that $v^*$ turns into
a vertex of type~2 in the graph $E^\up$. The vertex $v_*^\up$ on the other hand
has a new incoming edge which turns it into a vertex of type~3 if it was previously of type~1
and leaves it to be of type~2 (or in $\hat P \setminus \hat P^\circ$) otherwise.

We now turn to the proof of \eqref{e:integration}. Since the statement is trivial for 
$g = (\hat P,G,E)$ with $G$ empty, we only need to consider the case when it is non-empty.
For this, we first note that one
has the identity
\begin{equ}
K_\eps^{(\hat P,G,E)}(r) = K_\eps^{(\hat P',G',E)}(r \restr \hat P') K_\eps(r_{v_*} - r_{v^*})\;.
\end{equ}
It follows that we can write
\begin{equs}
J_\eps(s,t)(\hat P, G,E) &= \int_{[s,t]_{\hat P'}} K_\eps^{(\hat P',G',E)}(r)\int_{r_{v_*^\down}}^{r_{v_*^\up}} K_\eps(r' - r_{v^*}) \,dr'\,dr \\
&= \int_{[s,t]_{\hat P'}} K_\eps^{(\hat P',G',E)}(r)\big(\bar K_\eps(r_{v_*^\up} - r_{v^*})
- \bar K_\eps(r_{v_*^\down} - r_{v^*})\big)\,dr\;.
\end{equs}
In order to conclude that \eqref{e:integration} holds, it remains to note that 
$\bar K_\eps(0) = 0$, so that the corresponding terms vanish when $v^* \in \{v_*^\up,v_*^\down\}$.
The last statement follows immediately from the fact that $\CI$ decreases the cardinality of 
$G$ by one, unless $G$ is empty in which case it is the identity.
\end{proof}

The following lemma provides the a priori bounds on $\bar K_\eps$ necessary for our proofs.

\begin{lemma}\label{lem:propKbar}
The kernel $\bar K_\eps$ satisfies the following properties:
\begin{enumerate}
\item There exists a constant $C$ such that $\abs{\bar K_\eps(t)} \le C|t|^{-1/2}$,
uniformly over $\eps \in (0,1]$ and $|t| \le 1$.
\item There exists $c > 0$ such that, for any fixed $\delta \in (0,1]$,
one has 
\begin{equ}[e:intKbar]
\lim_{\eps \to 0} \int_{-\delta}^\delta \abs{\bar K_\eps(t)}^2 \,dt = c\;,\qquad 
\lim_{\eps \to 0} \int_{-\delta}^\delta \abs{\bar K_\eps(t)} \,dt = 0\;.
\end{equ}
\item For every $\alpha > 1/2$ and every $\delta > 0$ there exist  
$\eps > 0$ such that $\abs{\bar K_{\bar \eps}(t)} \le \delta |t|^{-\alpha}$
for all $\bar \eps < \eps$ and all $|t| \le 1$.
\end{enumerate}
\end{lemma}

\begin{proof}
This is a straightforward calculation based on the fact that $\bar K_\eps$ is proportional
to the convolution of $\rho_\eps$ with the function $t \mapsto C_\eps |t|^{2H-1}\sign t$.
\end{proof}

\begin{remark}
The second and third properties could of course be made more quantitative, but 
this would then lead us to always treat $H = 1/4$ as a special case which 
is something that we want to avoid.
\end{remark}

\begin{remark}
When $H < 1/4$, the constant $c$ is mollifier-dependent and given by
\begin{equ}
c = \int_{\R} (\rho \star F_H)(t)^2\,dt \;,
\end{equ}
where $F_H(t) = 2H|t|^{2H-1}\sign t$. When $H = 1/4$ however one has
$c = 4H^2$, independently of $\rho$. This is an instance of the usual fact that 
prefactors of logarithmic divergencies tend to be regularisation-independent.
\end{remark}

\subsection{Main estimates}

Given a linear poset $P$ and $E \subset P^\circ\times P^\circ$ such that $(\emptyset,E) \in \CG(P)$,
we say that $E$ is \textit{full} if every connected component of $E$ consists of a 
directed cycle of length $2$.
We furthermore say that $E$ is \textit{ordered} if the partial order on $P$ descends
to a partial order on the quotient $P_E$ of $P$ by the finest equivalence relation 
such that $e \sim f$ for any $(e,f) \in E$.
In other words, we want the transitive closure of the smallest 
relation $\le$ on $P_E$ such that $a\le b$ in $P$ implies $[a] \le [b]$ to be again a partial order.

Finally, we write $E^\le \subset E$ for those edges $(e,f)$ such that $e \le f$.
We are now in a position to state the following bound, which is the main working
horse for this section.

\begin{theorem}\label{thm:mainTechnicalfbm}
Let $P$ be a linear poset with $N = |P^\circ|$, let $E \subset P^\circ\times P^\circ$ be such that 
$(\emptyset,E) \in \CG(P)$, and write $g = (P,\emptyset,E)$. Let furthermore $T \in (0,1]$ and $[s,t]_P \subset [0,T]$. 
Then, one has the following.
\begin{enumerate}
\item There exists a constant $C$ such that $\abs{J_\eps(s,t)(g)} \le C T^{N/2}$, uniformly
over $\eps \in (0,1]$ and $T \in (0,1]$.
\item If either $E$ isn't full or $E$ is full but unordered, 
then $\lim_{\eps \to 0} J_\eps(s,t)(g) =0$.
\item If $E$ is full and ordered then, writing $k$ for the cardinality of $E^\le$, one has
\begin{equ}[e:wantedlimJeps]
\lim_{\eps \to 0} J_\eps(s,t)(g) = 2^{-k} (-c)^{N/2} \abs{[s,t]_{P_E}}\;,
\end{equ} 
where $c$ is as in \eqref{e:intKbar} and $\abs{\cdot}$ denotes Lebesgue measure on 
$\R^{P_E^\circ}$.
\end{enumerate}
\end{theorem}

\begin{remark}
If $E$ is unordered the definition of $P_E$ still makes sense, but it is only preordered
in the sense that there are $x \neq y \in P_E$ with $x \le y \le x$. In particular, 
every $r \in [s,t]_{P_E}$ must satisfy $r_x = r_y$ so that $\abs{[s,t]_{P_E}} = 0$
and \eqref{e:wantedlimJeps} still holds by point 2.
\end{remark}

\begin{proof}
To obtain the first bound, write $\abs{J_\eps(s,t)}$ for the valuation defined
like $J_\eps(s,t)$, but with all factors of $K_\eps$ and $\bar K_\eps$ replaced
by $|K_\eps|$ and $|\bar K_\eps|$.
If $v \in P^\circ$ is a vertex of degree $1$,
we then write $g' = (P',\emptyset,E')$ with $P' = P\setminus \{v\}$ and $E' = E \setminus \{e\}$ where $e$ is the unique edge
containing $v$. It is then straightforward that one again has $(E', \emptyset) \in \CG(P')$.

If such a vertex of degree $1$ exists, one then has the bound
\begin{equ}[e:boundIntone]
\abs{J_\eps(s,t)(g)} \le \abs{J_\eps(s,t)}(g)
\le CT^{1/2} \abs{J_\eps(s,t)}(g')\;,
\end{equ}
as a consequence of the first bound of Lemma~\ref{lem:propKbar}, which in particular implies that 
any integral of the form $\int_a^b K(t-c)\,dt$ with $a,b,c \in [0,T]$ is bounded
by $CT^{1/2}$ for some $C$. Iterating this bound, we can reduce ourselves to the case when 
$E$ consists of finitely many disjoint circles, each of which of size at least $2$, 
going through all elements of $P^\circ$. 

In fact, by freezing all other variables, we can reduce ourselves to the case of one 
single circle so that it remains to obtain 
a bound of the type $|I_k|  \lesssim T^{k/2}$ for any $k \ge 2$, where
\begin{equ}
I_k = \int_0^T\cdots \int_0^T \bar K_\eps(s_1-s_2)\cdots \bar K_\eps(s_{k-1}-s_k)\bar K_\eps(s_{k}-s_1)\,ds_1\cdots ds_k\;.
\end{equ}
By the first bound of \eqref{e:intKbar} with $\delta = T$ and Cauchy--Schwarz, we can bound the integral of the first and last factors over $s_1$ by a fixed constant.
This shows that $|I_k|$ is bounded by some fixed constant times
\begin{equ}
\int_0^T\cdots \int_0^T \bigl|\bar K_\eps(s_2-s_3)\cdots \bar K_\eps(s_{k-1}-s_k)\bigr|\,ds_2\cdots ds_k\;,
\end{equ}
which is in turn bounded by $T \big(\int_0^{2T} |\bar K_\eps(t)|\,dt\big)^{k-2}$
which is indeed of order $T^{k/2}$ as claimed.

We now turn to the second bound in the case when $E$ isn't full. If 
there exists a vertex $v$ of degree $1$, we can
use the third bound of Lemma~\ref{lem:propKbar} to conclude that, for any $\delta > 0$ and
$\kappa > 0$,
there exists $\eps$ such that, as in \eqref{e:boundIntone} (and with $g'$ defined
in the same way as it is there), one has
\begin{equ}
\abs{J_{\bar \eps}(s,t)}(g) \le \delta T^{1/2-\kappa} \abs{J_{\bar \eps}(s,t)}(g')\;,
\end{equ}
for any $\bar \eps < \eps$. It then follows from the first bound that  
$\lim_{\eps \to 0} \abs{J_{\bar \eps}(s,t)}(g) = 0$ as desired. If there is no vertex of 
degree $1$ then, since $E$ isn't full by assumption, it must contain at least one cycle of 
length $k \ge 3$. Proceeding as in the first bound, this yields a factor
$T \big(\int_0^{2T} |\bar K_\eps(t)|\,dt\big)^{k-2}$ which converges to $0$ as $\eps \to 0$ by
the second identity in \eqref{e:intKbar}.

Since the full but unordered case can be 
obtained in the same way as the third bound, we turn to that one first.
We start by observing that, when $G = \emptyset$ and $E$ is full, then
$N$ is necessarily even  and the
integrand of \eqref{e:defJeps} has constant sign $(-1)^{N/2}$. 
This is because each of the $N/2$ cycles of length $2$ leads to a factor of the 
type $\bar K_\eps(r_1-r_2)\bar K_\eps(r_2-r_1) = -\abs{\bar K_\eps(r_1-r_2)}^2$.
As a consequence, \eqref{e:wantedlimJeps} follows if we can show that
\begin{equ}[e:wantedlimJepsBis]
\lim_{\eps \to 0} \abs{J_\eps(s,t)}(g) = 2^{-k} c^{N/2} \abs{[s,t]_{P_E}}\;.
\end{equ}
We also note that the definition of the (pre)order on $P_E$ is such that, if we lift any
function $r \colon P_E \to \R$ to a function $\iota r$ on $P$ by setting 
$(\iota r)_u = r_{[u]}$, then $\iota r$ is monotone if and only if $r$ is monotone.
The fact that $E$ is ordered is then equivalent to the fact that there exist injective 
monotone functions on $P_E$. Write now $(s,t)_{P_E}$ for the subset of $[s,t]_{P_E}$
consisting of monotone functions that are injective on $P_E^\circ$. 

Every element $e$ of $P_E^\circ$ is an equivalence class consisting of two elements of $P$,
say $e_1$ and $e_2$.\footnote{The order here is completely arbitrary and our argument does
not depend on it.}
%If these are comparable, we order them so that $e_1 \le e_2$,
%i.e.\ $(e_1,e_2) \in E^\le$. 
Write $U_\delta$ for the set of functions $u\colon P_E^\circ \to (-\delta,\delta)$ with
the additional constraint that $u_e \ge 0$ if $e_1 \le e_2$ 
and $u_e \le 0$ if $e_2 \le e_1$. (If $e \not\in E^\le$ there is no constraint on 
the sign of $u_e$.)

We now claim that, given any $r \in (s,t)_{P_E}$, there exists 
$\delta > 0$ such that, for every $u \in U_\delta$, the function $(r|u)\colon P \to \R$ defined by 
\begin{equ}[e:defConcat]
(r|u)(e_1) = r(e)\;,\quad (r|u)(e_2) = r(e) + u(e)\;,
\end{equ}
and such that $(r|u)$ coincides with $s$ on $P^\bot$ and $t$ on $P^\top$ belongs
to $[s,t]_P$. For example, it suffices to take 
$\delta = \inf_{x\in P_E}\inf_{y\in P_E^\circ} |r_x-r_y|/2$. 
This is because 
we know that $(r|0) \in [s,t]_P$ and the boundary of $[s,t]_P$ consists of those
$r$ such that there exists $u\le v$ with $r_u = r_v$. Since $(r|0) \in [s,t]_P$ and
since $e_1$ and $e_2$ are incomparable for $(e_1,e_2) \not\in E^\le$ and satisfy
$e_1 \le e_2$ otherwise, it follows that we can move the coordinates of the
second argument by $\delta$ without leaving $[s,t]_P$, provided that we satisfy the 
ordering constraint on $E^\le$.

Writing $I(\eps,\delta) = \int_{-\delta}^\delta \bar K_\eps^2(u)\,du$, 
it therefore follows that
\begin{equ}
\abs{J_{\bar \eps}(s,t)}(g) \ge \int_{[s,t]_{P_E}} \int_{U_{\delta(r)}} |K_\eps^g(r|u)|\, du\,dr
\ge 2^{-k} \int_{[s,t]_{P_E}} I(\eps,\delta(r))^{N/2}\,dr\;.
\end{equ}
By Lemma~\ref{lem:propKbar}, $I(\eps,\delta(r))$ is bounded uniformly over $\eps,\delta \in (0,1]$
and satisfies $\lim_{\eps \to 0}I(\eps,\delta(r)) = c$ for every $\delta > 0$, so that the 
desired lower bound follows from the dominated convergence theorem.

To obtain the matching upper bound we first note that, in the same way as above,
\begin{equ}[e:badCase]
\lim_{\eps \to 0} \int_{[s,t]_{P_E}} \int_{U_{T}} K_\eps^g(r|u)\, du\,dr = 
2^{-k} c^{N/2} \abs{[s,t]_{P_E}}\;.
\end{equ}
This however is not sufficient since the image of $[s,t]_{P_E}\times U_T$ under
$(r,u) \mapsto (r|u)$ does not fully cover $[s,t]_{P}$. The reason is that while
every element in $[s,t]_{P}$ can uniquely be written in the form $(r|u)$, it does not
necessarily follow that $r \in [s,t]_{P_E}$.
However, if we take an arbitrary injective function $r \colon P_E \to \R$
which coincides with $s,t$ on $P^\bot$ and $P^\top$ and is such that $r \not \in [s,t]_{P_E}$,
then we know that $(r|0) \not \in [s,t]_{P_E}$. In particular, this shows again that there
exists $\delta > 0$ (depending on $r$) such that, for every $u$ such that 
$(r|u) \in [s,t]_{P}$, there exists some coordinate $f \in P_E^\circ$ such that $|u_f| \ge \delta$.
Write now $V_\delta^{(f)}$ for the set of maps $u \colon P_E^\circ \to [-T,T]$ such that 
$|u_f| \ge \delta$. It then follows that one has
\begin{equ}[e:upperBoundJeps]
\abs{J_{\bar \eps}(s,t)}(g) \le \int_{[s,t]_{P_E}} \int_{U_{T}} |K_\eps^g(r|u)|\, du\,dr
+ \sum_{f \in P_E^\circ} \int_{[s,t]_{P_E}^c} \int_{V_{\delta(r)}^{(f)}} |K_\eps^g(r|u)|\, du\,dr\;,
\end{equ}
where $[s,t]_{P_E}^c$ denotes the complement of $[s,t]_{P_E}$.
The first term converges to the desired bound by \eqref{e:badCase}. Regarding the second term,
each summand is bounded by 
\begin{equ}
\int_{[s,t]_{P_E}} I(\eps,T)^{N/2-1} \bigl(I(\eps,T) - I(\eps,\delta(r))\bigr)\,dr\;,
\end{equ}
which converges to zero by the dominated convergence theorem.

It now remains to treat the case when $E$ is full but unordered. It is then the case 
that for \textit{every} injective $r \colon P_E \to \R$ (compatible with $s,t$) there exists
$\delta > 0$ such that, for every $u$ such that 
$(r|u) \in [s,t]_{P}$, there exists some coordinate $f \in P_E^\circ$ such that $|u_f| \ge \delta$.
As a consequence, one has an upper bound analogous to \eqref{e:upperBoundJeps}, but with the first
term missing, which proves the claim.
\end{proof}

We conclude this section by deriving a consequence of Theorem~\ref{thm:mainTechnicalfbm} which 
covers all the cases of interest to us in a way that is much easier to verify,
namely in terms of graphs of the type $(P,G,\emptyset)$ which arise naturally in our
estimates, rather than the graphs of type $(P,\emptyset,E)$ for which the theorem applies.

Given a linear finite poset $P$ and a pairing $G$ of $P^\circ$, we say that $G$
is ``parallel'' if there exists a pairing $H$ of $G$ such that, for any
$\{\{e_1,e_2\}, \{f_1,f_2\}\} \in H$, one has (modulo relabelling one of the two edges)
\begin{equ}[e:pairedEdges]
f_1 \in \{e_1^\up, e_1^\down\}\;,\qquad f_2 \in \{e_2^\up, e_2^\down\}\;.
\end{equ}
We say that such a pair of edges is ``crossed'' if one can furthermore relabel
the elements in each of the two edges in such a way that
$f_1 = e_1^\up$ and $f_2 = e_2^\down$. 
The first two graphs in \eqref{e:CovarianceX2} below show one pair of uncrossed parallel 
edges and one pair of crossed parallel edges.

Note that if $G$ is parallel, then the
pairing $H$ is necessarily unique as a simple consequence of the finiteness of $G$.
(We can construct $H$ inductively by taking a minimal vertex $v$ not yet belonging to any
of the pairs of edges in $H$ and noting that the edge of $G$ containing $v$ can only possibly be 
paired with the unique edge containing its successor $v^\up$. This exhausts $G$ after finitely many steps.)
Given a parallel pairing $G$, we define the poset $\hat P$ as the quotient of $P$ by
the equivalence relation under which, for any pair of parallel edges as in \eqref{e:pairedEdges},
one has $e_1\sim f_1$ and $e_2 \sim f_2$. This is then naturally endowed with a full set
$E$ of directed edges which is such that, for any pair of edges in $H$, one
has $([e_1],[e_2]), ([e_2],[e_1]) \in E$.
We then have the following result, where we set $J(s,t) = \lim_{\eps \to 0} J_\eps(s,t)$
(which exists by Theorem~\ref{thm:mainTechnicalfbm}).

\begin{proposition}\label{prop:convMoments}
Let $P$ be a linear finite poset with $|P^\circ| = 2N$ and let $G$ be a pairing of $P^\circ$.
If $G$ is parallel then, defining $\hat P$ and $E$ as just discussed, one has
\begin{equs}
J(s,t)(P,G,\emptyset) &= 
(-1)^{\f N2-\ell} J(s,t)(\hat P,\emptyset,E) \\&= 
(-1)^\ell 2^{-k} c^{\f N2} \abs{[s,t]_{\hat P_E}} \;. \label{e:wantedLimit}
\end{equs} 
Here, $k$ denotes as before the cardinality of $E^{\le}$ while $\ell$
denotes the number of crossed pairs in $G$.
If $G$ is not parallel, then $J(s,t)(P,G,\emptyset) = 0$.
\end{proposition}

\begin{proof}
Denote by $\hat\CG(P) \subset \CG(P)$ the subsets of graphs $(\hat P,G,E)$ such that 
either $E$ contains an edge adjacent to $\hat P \setminus \hat P^\circ$ or $E$ contains
a connected component of size at least $2$ which isn't a two-cycle.
Since the action of $\CI$ always causes $E$ to grow, we see that 
$\CI \colon \hat\CG^\Z(P) \to \hat\CG^\Z(P)$ and therefore, by 
Theorem~\ref{thm:mainTechnicalfbm} and Proposition~\ref{prop:Integration},
that $\hat\CG(P) \subset \ker J(s,t)$.
In particular, both $\CI$ and $J(s,t)$ 
are well-defined on $\bar \CG^\Z(P) \eqdef \CG^\Z(P)/\hat \CG^\Z(P)$.

Assume now that the total order on $P$ used in the definition of $\CI$ is compatible
with the given partial order. Consider then $g = (\hat P, G,E) \in \CG(P)$ such that
$E$ contains only directed cycles of length $2$ and such that $G$ is non-empty. 
(In particular $g$ contains no vertex of type~3.)
Let furthermore $e = \{v^*,v_*\} \in G$ be the edge containing the smallest vertex $v^* \in \hat P^\circ$ of 
type~1 and let $f \in G$ be the edge containing the vertex $(v^*)^\up$ (and $f = \emptyset$ if no 
such edge exists). If $f \neq\emptyset$ and $f$ is parallel to $e$ (i.e.\ one has $f = \{(v^*)^\up,\hat v\}$ with
either
$\hat v = v_*^\up$ ``uncrossed''
or $\hat v = v_*^\down$ ``crossed''), then we define 
$g' = (\hat P', G', E')$ by setting
\begin{equ}
\hat P' = \hat P \setminus\{v_*, (v^*)^\up\}\;,\quad
G' = G\setminus\{e,f\}\;,\quad E' = E \cup \{(v^*,\hat v), (\hat v, v^*)\}\;.
\end{equ}
We claim that in $\bar \CG^\Z(P)$ one then has
\begin{equ}[e:wanted identity]
\CI^2 g = 
\left\{\begin{array}{cl}
	0 & \text{if $f = \emptyset$ or $f$ isn't parallel to $e$,} \\
	g' & \text{if the pair $(e,f)$ is crossed,}\\
	-g' & \text{if the pair $(e,f)$ is uncrossed.}
\end{array}\right.
\end{equ}
Since $\f N2-\ell$ is precisely the number of uncrossed pairs in $G$, this immediately implies
the first identity in \eqref{e:wantedLimit} (since $g'$ is of the same type as $g$, so this identity can
be iterated until $G$ is empty), whence the claim follows from Theorem~\ref{thm:mainTechnicalfbm}.

In order to show \eqref{e:wanted identity}, note that the effect of $\CI$ is to delete the node 
$v_*$ and to replace the edge $e \in G$ by $(v^*, \hat v)$ in $E$ with $\hat v \in \{v_*^\up,v_*^\down\}$,
yielding the two terms in \eqref{e:defCI}
(the second case being the one that comes with a minus sign). If either of these two nodes is of type~2, then this creates a connected component of $E$
with at least two edges, but this cannot be a cycle (since $v^*$ was of type~1) so such terms vanish
in $\bar \CG^\Z(P)$. If on the other hand $\hat v$ is of type~1 in $g$, then this turns 
into a node of type~3
in the corresponding term of $\CI g$. It follows that $\CI^2 g$ is obtained by 
considering the (unique) edge in $G$ of the form $f = \{\hat v, \tilde v\}$, deleting
the vertex $\tilde v$, and replacing $f$ by $(\hat v, \tilde v^\up)$ or $(\hat v, \tilde v^\down)$. Since this edge follows the edge $(v^*, \hat v)$, the only way in which we obtain 
a non-zero element of $\bar \CG^\Z(P)$ is if either $\tilde v^\up$ or $\tilde v^\down$
coincide with the node $v^*$ we started from.
Since $v^*$ was the minimal node of type~1 and since $\tilde v$ is necessarily of type~1,
it cannot be the case that $\tilde v^\up = v^*$, so the only way of getting a non-zero
element is to have $f$ parallel to $e$ and containing $(v^*)^\up$. 
Since $(v^*)^\up$ can only
be contained in one edge of $G$, we conclude that we have $\CI^2 = \pm g'$ when $f$ is
parallel to $e$ and $0$ otherwise. The correct sign is obtained by tracking the two cases.
\end{proof}

\subsection{A tool for convergence}
\label{sec:toolConv}

Before we turn to the convergence of the random rough paths $\X_\eps$, we 
provide an estimate for quantities of the same type $J_\eps(s,t)(g)$ as considered
in \eqref{e:defJeps}, but with the addition of a second type of edges that represent 
cutoff functions at a different scale $\delta > 0$. (One should think of $\delta \ll 1$, but
without any relation between the sizes of $\eps$ and $\delta$.) This section can be skipped at the 
first read, but its results will play a crucial role in identifying the limit
of the fourth-order component of $\X_\eps$.

More precisely, 
we fix once and for all a smooth increasing function $\chi \colon \R \to [0,1]$ with 
$\chi(0) = 0$ and $\chi(1) = 1$, and we set $\chi_\delta(t) = \chi(t/\delta)$.
We then show a variant of Theorem~\ref{thm:mainTechnicalfbm} which 
also allows for kernels of the type $(1-\chi_\delta)(t-s)$ and $\chi_\delta'(t-s)$.
Even though $1-\chi_\delta$ is not symmetric, it will be convenient to also represent it
by an unoriented edge which will be unambiguous since in our setting these kernels will always 
join two comparable vertices in our poset and the corresponding factor will always have an argument of the form
$r_u - r_v$ with $v \le u$.

To formalise this, we now assume that the set of unoriented edges $G$ comes endowed with a decomposition
$G = G_K \sqcup G_\chi$ and similarly for $E$. We also make the standing assumption
that edges in $G_K$ always connect vertices that are comparable.
It will then be convenient to denote edges $e$ in $G_\chi$ as 
$e = \{e_\down,e_\up\} \in G_\chi$ with $e_\down \le e_\up$. We will use the same
convention for edges $e \in E_\chi$. These edges furthermore come with their own orientation
and we introduce the notation $(-1)^e$ to denote whether these match or not. 
More precisely, given $e = (e_1,e_2) \in E_\chi$, we set $(-1)^e = 1$ if $e_1 \le e_2$
and $(-1)^e = -1$ otherwise. With all these notations at hand, we then set
similarly to \eqref{e:defJeps}
\begin{equ}[e:defJepsdelta]
J_{\eps,\delta}(s,t)( P, G,E) = \int_{[s,t]_{\hat P}} K_\eps^{( P,G_K,E_K)}(r)\,H_\delta^{( P,G_\chi,E_\chi)}(r) \,dr\;,
\end{equ}
where we wrote
\begin{equ}[e:defIntegrandJ]
H_\delta^{( P,G,E)}(r)
= \prod_{e \in G} (1-\chi_\delta)(r_{e_\up}-r_{e_\down})\prod_{e \in E}(-1)^e \chi_\delta'(r_{e_\up}-r_{e_\down})\;.
\end{equ}

Given a finite poset $P$, we now write similarly to before $\CG(P)$ 
for those graphs $(\hat P,G,E)$ such that $\hat P$ is a subset of $P$, 
$G = G_K \sqcup G_\chi$
(and similarly for $E$), such that edges in $G_\chi$ and $E_\chi$ are
between comparable elements, and
such their local structure around any vertex $v \in \hat P^\circ$ is of one
of the following types:
\begin{enumerate}
\item There is exactly one $g \in G_K$ with $v \in g$, at most one
$g \in G_\chi$ with $v \in g$, and no edge of $E$ touching $v$.
\item There is exactly one edge $e \in E$ pointing out of $v$, i.e.\ of the 
type $(v,v')$ for some $v' \in V$, and no edge in $G_K$ containing $v$.
If $e \in E_\chi$, then one furthermore has at least one edge in $E_K$ pointing
into $v$.
\item There is exactly one edge in $E$ pointing into $v$, no edge in $E$
pointing out of $v$, and exactly one edge in $G_K$ containing $v$.
\end{enumerate}
We again impose that there is at most one vertex of type~3
and that vertices in $V \setminus V^\circ$ only have edges in $E$ pointing into them.
We furthermore impose that distinct edges in $G_\chi$ are not comparable, i.e.\ if
$g_i = (g_i^\up, g_i^\down) \in G_\chi$ then $g_1^\up \le g_2^\up$ implies
$g_1 = g_2$.\footnote{The fact that $g_1^\down \le g_2^\down$ implies
$g_1 = g_2$ is of course equivalent.}
It is easy to check that Lemma~\ref{lem:components} still holds for such graphs.

We then define $\CG^\Z(P)$ as before as the $\Z$-module generated by $\CG(P)$, quotiented by
the submodule generated by graphs containing directed self-edges. Again, the valuation $J_{\eps,\delta}$
is well-defined on $\CG^\Z(P)$ as a consequence of the fact that one has $\bar K_\eps(0) = \chi'_\delta(0) = 0$.
In this setting, we can again define an integration by parts operator $\CJ \colon \CG^\Z(P) \to \CG^\Z(P)$ 
in a way quite similar to $\CI$,
but taking into account that we may not always be able to restrict our
integrations over nodes adjacent to a single edge.  
If $G_K$ is empty, then we do again set 
$\CJ(\hat P,G,E) = (\hat P,G,E)$. If not, let $v^*, v_* \in P^\circ$ be as
in the definition of $\CI$ on page~\pageref{p:defI}, but substituting $G_K$ for $G$.
We then define $\hat P' = \hat P \setminus \{v_*\}$,  $G' = G \setminus \{\{v^*,v_*\}\}$,
$E^\up = E_K \cup \{(v^*, v_*^\up)\}$ and $E^\down = E_K \cup \{(v^*, v_*^\down)\}$
in the same way as before.
However, while it is again the case that $v_*$ is of type~1 but there are now two possibilities,
depending on whether $v_*$ also belongs to an edge in $G_\chi$ or not.

If it does not, then we proceed exactly as before, setting 
\begin{equ}[e:defJ21]
\CJ(\hat P,G,E) = g^\up - g^\down\;,
\end{equ}
this time with $g^\up = (\hat P',G',E^\up \sqcup E_\chi)$.
If it does on the other hand, then there exists a unique $\bar v \in \hat P^\circ$ such that 
$e \eqdef \{v_*,\bar v\} \in G_\chi$ and we write $G'' = G' \setminus \{e\}$. In this case, we set
$G^\up = G'' \cup \{\{v_*^\up,\bar v\}\}$
and similarly for $G^\down$. 
Note that it may happen that $v_*^\up = \bar v$ (resp.\ $v_*^\down = \bar v$) in which case
we simply set $G^\up = G''$ (resp.\ $G^\down = G''$).\footnote{It may be easier to think of elements 
of $\CG(P)$ as being equivalent if the corresponding sets $G_\chi$ only differ by singletons.}
We then define $\CJ$ as
\begin{equ}
\CJ(\hat P,G,E) = g^\up - g^\down
 + (\hat P,G'',E^0)\;,\qquad\label{e:defJ22}
\end{equ}
where we set $E^0 = (E_K \cup \{(v^*,v_*)\}) \sqcup (E_\chi \cup \{(v_*,\bar v)\})$.
We now show that the analogue of \eqref{e:integration} still holds in this extended case.
The case when $v_*$ does not
belong to an edge in $G_\chi$ is treated exactly as in Proposition~\ref{prop:Integration}. 
In the other case, assuming that $v_* \ge \bar v$, we note that \eqref{e:defJepsdelta} can 
be written as 
\begin{equ}
J_{\eps,\delta}(s,t)(\hat P, G,E) = \int_{[s,t]_{\hat P'}} F_{\eps,\delta}(r)\int_{r_{v_*^\down}}^{r_{v_*^\up}} K_\eps(r' - r_{v^*})(1-\chi_\delta)(r' - r_{\bar v}) \,dr'\,dr\;,
\end{equ}
where the edge $g$ is as above and $F_{\eps,\delta}$ contains all the factors corresponding to 
edges in $G''$ and $E$. (In the case $v_* \le \bar v$ the argument of $\chi_\delta$ changes sign.)
We then integrate by parts, yielding
\begin{equs}
J_{\eps,\delta}(s,t)(\hat P, G,E) &= \int_{[s,t]_{\hat P'}} F_{\eps,\delta}(r)
 \bar K_\eps(r' - r_{v^*})(1-\chi_\delta)(r' - r_{\bar v})\Big|_{r_{v_*^\down}}^{r_{v_*^\up}}\,dr \\
&\quad + \int_{[s,t]_{\hat P'}} F_{\eps,\delta}(r)\int_{r_{v_*^\down}}^{r_{v_*^\up}} \bar K_\eps(r' - r_{v^*})\chi_\delta'(r' - r_{\bar v}) \,dr'\,dr\;,
\end{equs}
which can indeed be identified with the three terms appearing in \eqref{e:defJ22}.
Note that the additional sign appearing in the case $v_* \le \bar v$ is taken care of by
the factor $(-1)^e$ appearing in the definition \eqref{e:defIntegrandJ} of the integrand.

Since the action of $\CJ$ always decreases the number of elements in $G_K$ by one (unless 
it is empty), we see again that its successive applications stabilise, so that it suffices
to bound $J_{\eps,\delta}(s,t)(\CJ^\infty(\hat P,G,E))$ for the graphs appearing in 
\eqref{e:mainTermsDelta} and \eqref{e:extraTermsDelta}.
We then have the following analogue to Theorem~\ref{thm:mainTechnicalfbm}
where we say that $E$ is ``full'' if each of its connected components consists of a 
cycle containing exactly two edges in $E_G$ (but possibly additional edges in $E_\chi$).

\begin{theorem}\label{theo:boundepsdelta}
Let $P$ be a linear poset with $N = |P^\circ|$, let $E = E_K \sqcup E_\chi$ and $G_\chi$ be such that 
$g = (P,G_\chi,E) \in \CG(P)$. 
For any $T > 0$ there exists a function $o$ as above such that, for any fixed $s,t \le T$,
\begin{equ}
\abs{J_{\eps,\delta}(s,t)(g)} \le o(\eps) + o(\delta)\;,
\end{equ}
if $E$ isn't full or $G_\chi$ contains an element linking two distinct
connected components of $E$.
\end{theorem}

\begin{proof}
It suffices to show that $\abs{J_{\eps,\delta}(s,t)(g)} \le o(\eps)$
if $E$ isn't full and that $\abs{J_{\eps,\delta}(s,t)(g)} \le o(\delta) + o(\eps)$
if $E$ is full and $G_\chi$ links (at least) two of its connected components.

The proof of the first bound is virtually identical to that of the second
statement of Theorem~\ref{thm:mainTechnicalfbm}, noting that the factors 
$(1-\chi_\delta)$ can simply be bounded by one
(so edges in $G_\chi$ can be ignored) and that the integral of $|\chi_\delta'|$
is of order $1$.

To prove the second bound, we note that one has for $\abs{J_{\eps,\delta}(s,t)(g)}$ the 
bound \eqref{e:upperBoundJeps}, but with the integrand $K_\eps^g$ defined in 
such a way that it also includes the factors $(1-\chi_\delta)$. Since the second 
term in this expression is $o(\eps)$ we only need to bound the first term.
We now use the fact that 
the order used in the definition \eqref{e:defConcat} of $(r|u)$ was arbitrary.  We
can therefore make sure that there exist two elements $e,\bar e$ in $P_E^\circ$ 
(corresponding to the two cycles
linked by an edge in $G_\chi$) such that the corresponding elements $e_1$ and $\bar e_1$
appearing in \eqref{e:defConcat} satisfy $\{e_1,\bar e_2\} \in G_\chi$.
As a consequence, $(K_\eps^g)(r|u)$ contains a factor $(1-\chi_\delta)(r_{e}-r_{\bar e})$,
which implies that we can restrict the outer integral in the first term
of \eqref{e:upperBoundJeps} to the set of $r \in [s,t]_{P_E}$ such that 
$|r_{e}-r_{\bar e}| \le \delta$. This immediately implies that it is bounded by a quantity
of order $\delta$, as required.
\end{proof}

\subsection{Convergence}

We now have the tools required to show the main result of this section, namely Theorem~\ref{thm:mainRigorous} below, which is the precise formulation of Theorem~\ref{thm:mainfbm}.
In order to formulate our main result regarding the convergence of $\X_\eps$, we 
write $\cC_g^\alpha$ for the space of $\R^{m}$-valued 
geometric rough paths of regularity $\alpha $.
As a first ingredient, we show that the $\X_\eps$ as defined in \eqref{e:defLift} are tight in 
$\cC_g^\alpha$ for suitable $\alpha$.

\begin{proposition}\label{prop:tight}
For any $k \ge 1$ there exists a constant $C$ such that, for any $[s,t] \subset [0,1]$,
one has the bound
\begin{equ}
\E \abs{\X^{(k)}_\eps (s,t)}^2 \le C |t-s|^{k/2}\;,
\end{equ}
uniformly over $\eps \in (0,1]$. In particular, the law of $\X^{(k)}_\eps$ is tight in $\cC_g^\alpha$
for any $\alpha < 1/4$.
\end{proposition}

\begin{proof}
Let $P$ be the poset consisting of two copies of $\{0,\ldots,k+1\}$.
By \eqref{e:moments}, for any multiindex $I$ of length $k$, one has a collection 
$\CP^I$ of pairings of $P^\circ$ such that
\begin{equ}
\E \abs{\scal{\X^{(k)}_\eps (s,t),e_I}}^2 
= \sum_{G \in \CP^I} J_{\eps}(s,t)(P,G,\emptyset)
= \sum_{G \in \CP^I} J_{\eps}(s,t)\big(\CI^\infty(P,G,\emptyset)\big)\;.
\end{equ}
Since all the summands appearing in $\CI^\infty(P,G,\emptyset)$
are of the form $(\hat P,\emptyset,E)$ with $(E, \emptyset) \in \CG(\hat P)$
and $|\hat P^\circ| = k$, the claim follows from the first claim of Theorem~\ref{thm:mainTechnicalfbm}
(note that, by stationarity of increments, one can set $s=0$ and $t=T$ without loss of generality).
\end{proof}

As a next step, we show that the components of $\X_\eps^{(2)}$ converge in law
to independent Wiener processes $W_{ij}$ subject to the antisymmetry condition
$W_{ji} = -W_{ij}$. Here and below, given a function $X$ of one variable, we write $\delta X$ as a 
shorthand for the increment process $\delta X (s,t) = X(t) - X(s)$.

\begin{proposition}
Let $W_{ij}$ for $i< j$ with $i,j \in \{1,\ldots,m\}$ be i.i.d.\ standard Wiener 
processes, let $W_{ji} = -W_{ij}$, and set $W_{ii} = 0$. Then, for any $\alpha \in (1/5,1/4)$,
$\X_\eps$ converges in law in $\cC_g^\alpha$
to the limit $\X$ given as follows. One has $\X^{(1)} = 0$, $\X^{(3)} = 0$,
$\X^{(2)}_{ij}(s,t) = \sigma \delta W_{ij}(s,t)$, and 
$\X^{(4)}_{ijk\ell}(s,t)= \sigma^2 \int_s^t \delta W_{ij}(s,r)\circ dW_{k\ell}(r)$.
Here, $\sigma = \sqrt c$ with $c$ as in \eqref{e:intKbar}.
\end{proposition}

\begin{remark}
In fact, the Wiener processes $W_{ij}$ are independent of the fractional Brownian motions
$B_i$ in the sense of stable convergence \cite{StableMall}.
\end{remark}

\begin{proof}
Since we already have tightness by Proposition~\ref{prop:tight}, it
suffices to show convergence of finite-dimensional distributions.
The fact that $\lim_{\eps \to 0}\E \abs{\X_\eps^{(k)}(s,t)}^2 = 0$ for any
$s < t$ and $k$ odd follows immediately from Theorem~\ref{thm:mainTechnicalfbm}
since this quantity can be written as a finite linear combination of terms of the type
$J(s,t)(P,\emptyset,E)$ with $|P^\circ| = k$ and $E$ cannot be full unless $k$ is
even.

Regarding $\X^{(2)}$, the special case $m=2$ was treated in \cite{unterberger2008central},
but since we need the general case (which doesn't appear to follow in a straightforward way)
and since our proof is much shorter we give it here.
Our strategy is to apply the fourth moment theorem \cite{FourthMoment}
which states that, when considering random variables belonging to a Wiener 
chaos of finite order, convergence of covariances and fourth mixed moments is sufficient
to guarantee convergence of finite-dimensional distributions to a Gaussian limit.
This however is a relatively straightforward consequence of Proposition~\ref{prop:convMoments}.
Indeed, $\E \X^{(2)}_{\eps;ij}(s_1,t_1)\X^{(2)}_{\eps;k\ell}(s_2,t_2)$ is 
given by $J_\eps(s,t)(g)$ where, pictorially, $g = (P,G,\emptyset)$ is given by
\begin{equ}[e:CovarianceX2]
g = 
\delta_{ik}\delta_{j\ell}
\begin{tikzpicture}[baseline=1.3cm]
\node[dot,label=below:{$s_1$}] (1) at (0,0.4) {};
\node[dot,label=below:{$s_2$}] (1') at (1,0.4) {};
\node[dot,label=above:{$t_1$}] (4) at (0,2.6) {};
\node[dot,label=above:{$t_2$}] (4') at (1,2.6) {};
\node[dot,label=left:{$i$}] (2) at (0,1) {};
\node[dot,label=left:{$j$}] (3) at (0,2) {};
\node[dot,label=right:{$k$}] (2') at (1,1) {};
\node[dot,label=right:{$\ell$}] (3') at (1,2) {};

\draw[black!15] (1) -- (2) -- (3) -- (4);
\draw[black!15] (1') -- (2') -- (3') -- (4');

\draw[very thick,darkblue] (2) -- (2');
\draw[very thick,darkblue] (3) -- (3');
\end{tikzpicture}
+
\delta_{i\ell}\delta_{jk}
\begin{tikzpicture}[baseline=1.3cm]
\node[dot,label=below:{$s_1$}] (1) at (0,0.4) {};
\node[dot,label=below:{$s_2$}] (1') at (1,0.4) {};
\node[dot,label=above:{$t_1$}] (4) at (0,2.6) {};
\node[dot,label=above:{$t_2$}] (4') at (1,2.6) {};
\node[dot,label=left:{$i$}] (2) at (0,1) {};
\node[dot,label=left:{$j$}] (3) at (0,2) {};
\node[dot,label=right:{$k$}] (2') at (1,1) {};
\node[dot,label=right:{$\ell$}] (3') at (1,2) {};

\draw[black!15] (1) -- (2) -- (3) -- (4);
\draw[black!15] (1') -- (2') -- (3') -- (4');

\draw[very thick,darkblue] (2) -- (3');
\draw[very thick,darkblue] (3) -- (2');
\end{tikzpicture}
+
\delta_{ij}\delta_{k\ell}
\begin{tikzpicture}[baseline=1.3cm]
\node[dot,label=below:{$s_1$}] (1) at (0,0.4) {};
\node[dot,label=below:{$s_2$}] (1') at (1,0.4) {};
\node[dot,label=above:{$t_1$}] (4) at (0,2.6) {};
\node[dot,label=above:{$t_2$}] (4') at (1,2.6) {};
\node[dot,label=left:{$i$}] (2) at (0,1) {};
\node[dot,label=left:{$j$}] (3) at (0,2) {};
\node[dot,label=right:{$k$}] (2') at (1,1) {};
\node[dot,label=right:{$\ell$}] (3') at (1,2) {};

\draw[black!15] (1) -- (2) -- (3) -- (4);
\draw[black!15] (1') -- (2') -- (3') -- (4');

\draw[very thick,darkblue] (2) -- (3);
\draw[very thick,darkblue] (3') -- (2');
\end{tikzpicture}\;.
\end{equ}
The light gray lines denote the Hasse diagram of the poset $P$ while the pairing 
$G$ is shown in blue.
Note that the  pairings for the first two terms are parallel (with the second one
consisting of a crossed pair), while the third one is not, so that, by the 
first equality of Proposition~\ref{prop:convMoments}, one has
$J(s,t)(g) = J(s,t)(g')$ with
\begin{equ}
g' = ( \delta_{i\ell}\delta_{jk} - \delta_{ik}\delta_{j\ell})
\begin{tikzpicture}[baseline=0.6cm]
\node[dot,label=below:{$s_1$}] (1) at (0,0) {};
\node[dot,label=below:{$s_2$}] (1') at (1,0) {};
\node[dot,label=above:{$t_1$}] (4) at (0,1.4) {};
\node[dot,label=above:{$t_2$}] (4') at (1,1.4) {};
\node[dot] (2) at (0,.7) {};
\node[dot] (2') at (1,.7) {};

\draw[black!15] (1) -- (2) -- (4);
\draw[black!15] (1') -- (2') -- (4');

\draw[thick,lightblue,->] (2) to[bend left=20] (2');
\draw[thick,lightblue,->] (2') to[bend left=20] (2);
\end{tikzpicture}
\end{equ}
where now directed edges in $E$ are indicated as light blue arrows.
It thus follows from the second equality of Proposition~\ref{prop:convMoments} that
for $i \neq j$
\begin{equ}
\lim_{\eps \to 0} \E \X^{(2)}_{\eps;ij}(s_1,t_1)\X^{(2)}_{\eps;k\ell}(s_2,t_2)
= 
\left\{\begin{array}{cl}
	c \abs{[s_1,t_1]\cap [s_2,t_2]} & \text{if $(ij) = (k\ell)$,} \\
	-c \abs{[s_1,t_1]\cap [s_2,t_2]} & \text{if $(ij) = (\ell k)$,} \\
	0 & \text{otherwise,}
\end{array}\right.
\end{equ} 
with the limit also vanishing when $i=j$,
which is precisely the covariance structure described in the claim.

To complete the proof of Gaussianity of the limit, it remains to show that all 
joint cumulants of order $4$ of terms of the type $\X^{(2)}_{\eps;ij}(s,t)$
vanish as $\eps \to 0$. For this, let $P_0$ denote the totally ordered poset with $4$
elements. Then, the joint cumulants are given by $J_\eps(s,t)(g)$ where $g$ is a 
linear combination of pairings of $P^\circ$ with $P = P_0 \sqcup P_0 \sqcup P_0 \sqcup P_0$, which 
cannot be written as the disjoint union of two pairings of $(P_0 \sqcup P_0)^\circ$.
Since it is easy to see that every parallel pairing of $P^\circ$ can be written 
as such a disjoint union, the vanishing of the cumulants as $\eps \to 0$
follows from Proposition~\ref{prop:convMoments}.

By Skorokhod's representation theorem, we can (and will) therefore assume from now on
that $\X^{(2)}_{\eps}$ converges in probability in $\CC^\alpha$ to 
$\X^{(2)} = \sigma \delta W$ as described in the statement. 
It remains to show that, for any fixed $s<t$, $\X^{(4)}_\eps$ converges in probability to the second order iterated 
integrals of the Wiener processes $W_{ij}$. 
Note first that one has
\begin{equ}[e:doubleIndex]
\lim_{\eps \to 0} \X^{(4)}_{\eps;iijk}
= \lim_{\eps \to 0} \X^{(4)}_{\eps;jkii} = 0\;,
\end{equ}
for all indices $i,j,k$.
This is an immediate consequence of Proposition~\ref{prop:convMoments}, noting
that for every parallel pairing appearing in the graphical representation
of $\E (\X^{(4)}_{\eps;iijk})^2$ for which the two edges containing the 
vertices associated to the two factors $\eta_{\eps;i}$ appearing in one of the 
two copies of  $\X^{(4)}_{\eps;iijk}$ are crossed there is a corresponding pairing in which 
they are uncrossed, with both contributions cancelling out.

As a consequence, we also deduce that one has
\begin{equ}
\lim_{\eps \to 0} \big(\X^{(4)}_{\eps;ijij} - \f12(\X^{(2)}_{\eps;ij})^2\big) = 0\;,
\end{equ}
so that, as a consequence of the convergence of $\X^{(2)}_{\eps}$, one has
indeed 
$\lim_{\eps \to 0} \X^{(4)}_{\eps;ijij} = \X^{(4)}_{ijij}
\eqdef \f{\sigma^2}2 (\delta W_{ij})^2 = \sigma^2 \int_s^t \delta W_{ij}(s,r)\circ dW_{ij}(r)$
This is an immediate consequence of the shuffle identity
\begin{equ}
(\X^{(2)}_{\eps;ij})^2
= 2\X^{(4)}_{\eps;ijij} + 4\X^{(4)}_{\eps;iijj}\;,
\end{equ}
combined with \eqref{e:doubleIndex}.

We also note that one has the identities 
\begin{equ}[e:antisym]
\lim_{\eps \to 0} \big(\X^{(4)}_{\eps;ijk\ell}(s,t)
+ \X^{(4)}_{\eps;jik\ell}(s,t)\big) = 
\lim_{\eps \to 0} \big(\X^{(4)}_{\eps;k\ell ij}(s,t)
+ \X^{(4)}_{\eps;k\ell ji}(s,t)\big) = 0\;,
\end{equ}
which can be shown in virtually the same way as \eqref{e:doubleIndex}.
As a consequence, it remains to show that $\X^{(4)}_{\eps;ijk\ell}(s,t)$ has the claimed
limit in the cases when $i,j,k,\ell$ are all distinct or when $j=k$ and $i,j,\ell$ are distinct.

Note now that if we set $\B^{(1)}_\eps = \X^{(2)}_\eps$ and define
\begin{equ}
\B^{(2)}_\eps(s,t) = \int_s^t \B^{(1)}_\eps(s,r)\otimes \d_r \B^{(1)}_\eps(s,r)\,dr\;,
\end{equ}
then one has the shuffle identity
\begin{equ}[e:idenIntegralEps]
\B^{(2)}_{\eps; ijk\ell} = \X^{(4)}_{\eps; ijk\ell} + \X^{(4)}_{\eps; ikj\ell} + \X^{(4)}_{\eps; kij\ell}\;.
\end{equ}
As a consequence of \eqref{e:antisym} the last two terms cancel out in the 
limit $\eps \to 0$ so that it remains to show that 
$\B^{(2)}_{\eps; ijk\ell}$ converges to $\sigma^2$ times the Stratonovich integral of $W_{ij}$ against $W_{k\ell}$. The reason for considering $\B^{(2)}_{\eps}$ rather than simply
$\X^{(4)}_{\eps}$ is that the former is already of the form of an integral of two processes 
converging towards Wiener processes against each other.

Let now $\chi_\delta$ be as in Section~\ref{sec:toolConv} and let
\begin{equ}[e:defMollifiedB2]
\B^{(2)}_{\eps,\delta}(s,t) = \int_s^t \int_s^r\chi_\delta(r-u) \d_u\B^{(1)}_\eps(s,u)\otimes \d_r \B^{(1)}_\eps(s,r)\,du\,dr\;.
\end{equ}
It is straightforward to verify the following properties (we fix $s,t$ throughout).
\begin{enumerate}
\item For any fixed $\delta > 0$, one has $\lim_{\eps \to 0} \B^{(2)}_{\eps,\delta}
= \tilde \B^{(2)}_{\delta}$, where
\begin{equ}
\tilde \B^{(2)}_{\delta}(s,t) = \int_s^t \int_s^r\chi_\delta(r-u)\, dW(u)\otimes dW(r)\;.
\end{equ}
This is a consequence of the convergence $\B_\eps^{(1)} \to W$ in $\CC^\alpha$  
obtained in the first part, combined with the fact that the inner integral in 
\eqref{e:defMollifiedB2} yields a smooth function of $r$ for every continuous function
$\B_\eps^{(1)}$, so that the map $\B_\eps^{(1)} \mapsto \B_{\eps,\delta}^{(2)}$ is
continuous on $\CC^\alpha$.
\item One has $\B^{(2)} \eqdef \lim_{\delta \to 0} \tilde \B^{(2)}_{\delta} = \X^{(4)}$ with 
$\X^{(4)}$ as described in the statement of the proposition. 
\end{enumerate}
In order to conclude, it thus remains to show that
the convergence $\lim_{\delta \to 0} \B^{(2)}_{\eps,\delta} = \B^{(2)}_{\eps}$ 
has good enough uniformity properties. In fact, we will show the sufficient fact that
there exist a function $o \colon (0,1] \to (0,1]$ with $\lim_{\eps \to 0} o(\eps) = 0$
such that
\begin{equ}[e:wantedBound]
 \E \|\tilde \B^{(2)}_{\eps,\delta}(s,t)\bigr\|^2 \le o(\eps) + o(\delta)\;,
\end{equ}
where $\tilde \B^{(2)}_{\eps,\delta} = \B^{(2)}_{\eps} - \B^{(2)}_{\eps,\delta}$ is given by the 
same expression as \eqref{e:defMollifiedB2}, but with $\chi_\delta$ replaced by $1-\chi_\delta$.

Before we proceed, let us just show how \eqref{e:wantedBound} implies the claim. For every $\delta > 0$ one has (all processes being evaluated at a fixed pair of times $(s,t)$):
\begin{equ}
 \E \|\X^{(4)} - \B^{(2)}_{\eps}\|^2
 \lesssim  \E \|\X^{(4)} - \tilde \B^{(2)}_{\delta}\|^2 + \E \|\tilde \B^{(2)}_{\delta} - \B^{(2)}_{\eps,\delta}\|^2 +  \E \|\tilde \B^{(2)}_{\eps,\delta}\|^2
\end{equ}
Given some $\bar \eps > 0$, we first choose $\delta$ small enough so that 
$\E \|\X^{(4)} - \tilde \B^{(2)}_{\delta}\|^2 \le \bar \eps$ (which is 
possible by the first property) and $o(\delta) < \bar \eps$.
We then choose $\eps$ small enough so that $o(\eps) < \bar \eps$ and 
$\E \|\tilde \B^{(2)}_{\delta} - \B^{(2)}_{\eps,\delta}\|^2 \le \bar \eps$ (which 
is possible by the second property), so that the 
right hand side of the displayed equation is smaller than $4\bar \eps$.

In the notations of Section~\ref{sec:toolConv}, it follows similarly to
\eqref{e:idenIntegralEps} that one has the identity
\begin{equ}
\tilde\B^{(2)}_{\eps,\delta; ijk\ell} = \X^{(4;24)}_{\eps,\delta; ijk\ell} + \X^{(4;34)}_{\eps; ikj\ell} + \X^{(4;34)}_{\eps; kij\ell}\;,
\end{equ}
where we set
\begin{equ}
\X^{(4;uv)}_{\eps,\delta}
= \int_s^t \!\int_s^{r_4}\!\int_s^{r_3} \!\int_s^{r_2}  (1-\chi_\delta)(r_u-r_v)\,\xi_\eps(r_1)\otimes \ldots \otimes \xi_\eps(r_k)\,dr_1\ldots dr_4\;.
\end{equ}

It follows that, for $i,j,k,\ell$ all distinct, one has the identity (with the arguments
$(s,t)$ fixed and omitted for clarity)
\begin{equs}
\E |\tilde \B^{(2)}_{\eps,\delta;ijk\ell}|^2
&\le 
2\E |\X^{(4;24)}_{\eps,\delta; ijk\ell}|^2
+
2\E |\X^{(4;34)}_{\eps; ikj\ell} + \X^{(4;34s)}_{\eps; kij\ell}|^2 \\
&= 
J_{\eps,\delta} \left(
2
\begin{tikzpicture}[baseline=1.15cm]
\node[dot] (1) at (0,0) {};
\node[dot] (1') at (1,0) {};
\node[dot] (6) at (0,2.5) {};
\node[dot] (6') at (1,2.5) {};
\node[dot] (2) at (0,0.5) {};
\node[dot] (3) at (0,1) {};
\node[dot] (4) at (0,1.5) {};
\node[dot] (5) at (0,2) {};
\node[dot] (2') at (1,0.5) {};
\node[dot] (3') at (1,1) {};
\node[dot] (4') at (1,1.5) {};
\node[dot] (5') at (1,2) {};

\draw[black!15] (1) -- (2) -- (3) -- (4) -- (5) -- (6);
\draw[black!15] (1') -- (2') -- (3') -- (4') -- (5') -- (6');

\draw[very thick,darkblue] (2) -- (2');
\draw[very thick,darkblue] (3) -- (3');
\draw[very thick,darkblue] (4) -- (4');
\draw[very thick,darkblue] (5) -- (5');

\draw[thick,darkred] (3) to[bend left=30] (5);
\draw[thick,darkred] (3') to[bend right=30] (5');

\end{tikzpicture}
+
4
\begin{tikzpicture}[baseline=1.15cm]
\node[dot] (1) at (0,0) {};
\node[dot] (1') at (1,0) {};
\node[dot] (6) at (0,2.5) {};
\node[dot] (6') at (1,2.5) {};
\node[dot] (2) at (0,0.5) {};
\node[dot] (3) at (0,1) {};
\node[dot] (4) at (0,1.5) {};
\node[dot] (5) at (0,2) {};
\node[dot] (2') at (1,0.5) {};
\node[dot] (3') at (1,1) {};
\node[dot] (4') at (1,1.5) {};
\node[dot] (5') at (1,2) {};

\draw[black!15] (1) -- (2) -- (3) -- (4) -- (5) -- (6);
\draw[black!15] (1') -- (2') -- (3') -- (4') -- (5') -- (6');

\draw[very thick,darkblue] (2) -- (2');
\draw[very thick,darkblue] (3) -- (3');
\draw[very thick,darkblue] (4) -- (4');
\draw[very thick,darkblue] (5) -- (5');

\draw[thick,darkred] (4) to[bend left=30] (5);
\draw[thick,darkred] (4') to[bend right=30] (5');
\end{tikzpicture}
+
4
\begin{tikzpicture}[baseline=1.15cm]
\node[dot] (1) at (0,0) {};
\node[dot] (1') at (1,0) {};
\node[dot] (6) at (0,2.5) {};
\node[dot] (6') at (1,2.5) {};
\node[dot] (2) at (0,0.5) {};
\node[dot] (3) at (0,1) {};
\node[dot] (4) at (0,1.5) {};
\node[dot] (5) at (0,2) {};
\node[dot] (2') at (1,0.5) {};
\node[dot] (3') at (1,1) {};
\node[dot] (4') at (1,1.5) {};
\node[dot] (5') at (1,2) {};

\draw[black!15] (1) -- (2) -- (3) -- (4) -- (5) -- (6);
\draw[black!15] (1') -- (2') -- (3') -- (4') -- (5') -- (6');

\draw[very thick,darkblue] (2) -- (3');
\draw[very thick,darkblue] (3) -- (2');
\draw[very thick,darkblue] (4) -- (4');
\draw[very thick,darkblue] (5) -- (5');

\draw[thick,darkred] (4) to[bend left=30] (5);
\draw[thick,darkred] (4') to[bend right=30] (5');
\end{tikzpicture}
\right)\label{e:mainTermsDelta}
\end{equs}
where edges in $G_K$ are drawn in blue as before while edges in
$G_\chi$ are drawn in red.
One furthermore has $\E |\tilde \B^{(2)}_{\eps,\delta;ijj\ell}|^2
\le \E |\tilde \B^{(2)}_{\eps,\delta;ijk\ell}|^2
+ 
J_{\eps,\delta}(g)$ with
\begin{equ}[e:extraTermsDelta]
g = 2
\begin{tikzpicture}[baseline=1.15cm]
\node[dot] (1) at (0,0) {};
\node[dot] (1') at (1,0) {};
\node[dot] (6) at (0,2.5) {};
\node[dot] (6') at (1,2.5) {};
\node[dot] (2) at (0,0.5) {};
\node[dot] (3) at (0,1) {};
\node[dot] (4) at (0,1.5) {};
\node[dot] (5) at (0,2) {};
\node[dot] (2') at (1,0.5) {};
\node[dot] (3') at (1,1) {};
\node[dot] (4') at (1,1.5) {};
\node[dot] (5') at (1,2) {};

\draw[black!15] (1) -- (2) -- (3) -- (4) -- (5) -- (6);
\draw[black!15] (1') -- (2') -- (3') -- (4') -- (5') -- (6');

\draw[very thick,darkblue] (2) -- (2');
\draw[very thick,darkblue] (3) -- (4);
\draw[very thick,darkblue] (3') -- (4');
\draw[very thick,darkblue] (5) -- (5');

\draw[thick,darkred] (3) to[bend left=30] (5);
\draw[thick,darkred] (3') to[bend right=30] (5');

\end{tikzpicture}
+
2
\begin{tikzpicture}[baseline=1.15cm]
\node[dot] (1) at (0,0) {};
\node[dot] (1') at (1,0) {};
\node[dot] (6) at (0,2.5) {};
\node[dot] (6') at (1,2.5) {};
\node[dot] (2) at (0,0.5) {};
\node[dot] (3) at (0,1) {};
\node[dot] (4) at (0,1.5) {};
\node[dot] (5) at (0,2) {};
\node[dot] (2') at (1,0.5) {};
\node[dot] (3') at (1,1) {};
\node[dot] (4') at (1,1.5) {};
\node[dot] (5') at (1,2) {};

\draw[black!15] (1) -- (2) -- (3) -- (4) -- (5) -- (6);
\draw[black!15] (1') -- (2') -- (3') -- (4') -- (5') -- (6');

\draw[very thick,darkblue] (2) -- (2');
\draw[very thick,darkblue] (3) -- (4);
\draw[very thick,darkblue] (3') -- (4');
\draw[very thick,darkblue] (5) -- (5');

\draw[thick,darkred] (4) to[bend left=30] (5);
\draw[thick,darkred] (4') to[bend right=30] (5');
\end{tikzpicture}
+
2
\begin{tikzpicture}[baseline=1.15cm]
\node[dot] (1) at (0,0) {};
\node[dot] (1') at (1,0) {};
\node[dot] (6) at (0,2.5) {};
\node[dot] (6') at (1,2.5) {};
\node[dot] (2) at (0,0.5) {};
\node[dot] (3) at (0,1) {};
\node[dot] (4) at (0,1.5) {};
\node[dot] (5) at (0,2) {};
\node[dot] (2') at (1,0.5) {};
\node[dot] (3') at (1,1) {};
\node[dot] (4') at (1,1.5) {};
\node[dot] (5') at (1,2) {};

\draw[black!15] (1) -- (2) -- (3) -- (4) -- (5) -- (6);
\draw[black!15] (1') -- (2') -- (3') -- (4') -- (5') -- (6');

\draw[very thick,darkblue] (2) to[bend left=30] (4);
\draw[very thick,darkblue] (3) -- (3');
\draw[very thick,darkblue] (2') to[bend right=30] (4');
\draw[very thick,darkblue] (5) -- (5');

\draw[thick,darkred] (4) to[bend left=30] (5);
\draw[thick,darkred] (4') to[bend right=30] (5');
\end{tikzpicture}
+
4
\begin{tikzpicture}[baseline=1.15cm]
\node[dot] (1) at (0,0) {};
\node[dot] (1') at (1,0) {};
\node[dot] (6) at (0,2.5) {};
\node[dot] (6') at (1,2.5) {};
\node[dot] (2) at (0,0.5) {};
\node[dot] (3) at (0,1) {};
\node[dot] (4) at (0,1.5) {};
\node[dot] (5) at (0,2) {};
\node[dot] (2') at (1,0.5) {};
\node[dot] (3') at (1,1) {};
\node[dot] (4') at (1,1.5) {};
\node[dot] (5') at (1,2) {};

\draw[black!15] (1) -- (2) -- (3) -- (4) -- (5) -- (6);
\draw[black!15] (1') -- (2') -- (3') -- (4') -- (5') -- (6');

\draw[very thick,darkblue] (2) to[bend left=30] (4);
\draw[very thick,darkblue] (3) -- (2');
\draw[very thick,darkblue] (3') to[bend right=30] (4');
\draw[very thick,darkblue] (5) -- (5');

\draw[thick,darkred] (4) to[bend left=30] (5);
\draw[thick,darkred] (4') to[bend right=30] (5');
\end{tikzpicture}\;.
\end{equ}
We will show that all the terms appearing in \eqref{e:mainTermsDelta}
and \eqref{e:extraTermsDelta} are bounded by $o(\eps) + o(\delta)$, except for
the last two terms in \eqref{e:mainTermsDelta} that have to be combined in order
to take advantage of cancellations.

It follows immediately that the terms in \eqref{e:extraTermsDelta}
are all bounded by $o(\eps)$ since no cycle with two elements of $E_\chi$ can be generated by
successive applications of $\CJ$.
Regarding \eqref{e:mainTermsDelta}, the first term is bounded by $o(\delta) + o(\eps)$ since the 
only ``full'' term that can be obtained by successive applications of $\CJ$ is
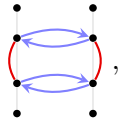
\begin{equ}
\begin{tikzpicture}[baseline=0.6cm]
\node[dot] (1) at (0,0) {};
\node[dot] (1') at (1,0) {};
\node[dot] (4) at (0,1.4) {};
\node[dot] (4') at (1,1.4) {};
\node[dot] (2) at (0,.4) {};
\node[dot] (2') at (1,.4) {};
\node[dot] (3) at (0,1) {};
\node[dot] (3') at (1,1) {};

\draw[black!15] (1) -- (2) -- (3) -- (4);
\draw[black!15] (1') -- (2') -- (3') -- (4');

\draw[thick,lightblue,->] (2) to[bend left=20] (2');
\draw[thick,lightblue,->] (2') to[bend left=20] (2);
\draw[thick,lightblue,->] (3) to[bend left=20] (3');
\draw[thick,lightblue,->] (3') to[bend left=20] (3);

\draw[thick,darkred] (2) to[bend left=30] (3);
\draw[thick,darkred] (2') to[bend right=30] (3');
\end{tikzpicture}\;,
\end{equ}
which yields a contribution at most $o(\delta)$ by Theorem~\ref{theo:boundepsdelta}.

It remains to deal with the last two terms in \eqref{e:mainTermsDelta}. It turns out that 
these terms are \textit{not} separately bounded by $o(\eps)+o(\delta)$, but their sum is due
to cancellations. This is because, if we choose our fixed total order on the vertices 
such that the smallest vertex is the second one from the bottom of the left column, 
then one has
\begin{equ}
\CJ \left(
\begin{tikzpicture}[baseline=1.15cm]
\node[dot] (1) at (0,0) {};
\node[dot] (1') at (1,0) {};
\node[dot] (6) at (0,2.5) {};
\node[dot] (6') at (1,2.5) {};
\node[dot] (2) at (0,0.5) {};
\node[dot] (3) at (0,1) {};
\node[dot] (4) at (0,1.5) {};
\node[dot] (5) at (0,2) {};
\node[dot] (2') at (1,0.5) {};
\node[dot] (3') at (1,1) {};
\node[dot] (4') at (1,1.5) {};
\node[dot] (5') at (1,2) {};

\draw[black!15] (1) -- (2) -- (3) -- (4) -- (5) -- (6);
\draw[black!15] (1') -- (2') -- (3') -- (4') -- (5') -- (6');

\draw[very thick,darkblue] (2) -- (2');
\draw[very thick,darkblue] (3) -- (3');
\draw[very thick,darkblue] (4) -- (4');
\draw[very thick,darkblue] (5) -- (5');

\draw[thick,darkred] (4) to[bend left=30] (5);
\draw[thick,darkred] (4') to[bend right=30] (5');
\end{tikzpicture}
+
\begin{tikzpicture}[baseline=1.15cm]
\node[dot] (1) at (0,0) {};
\node[dot] (1') at (1,0) {};
\node[dot] (6) at (0,2.5) {};
\node[dot] (6') at (1,2.5) {};
\node[dot] (2) at (0,0.5) {};
\node[dot] (3) at (0,1) {};
\node[dot] (4) at (0,1.5) {};
\node[dot] (5) at (0,2) {};
\node[dot] (2') at (1,0.5) {};
\node[dot] (3') at (1,1) {};
\node[dot] (4') at (1,1.5) {};
\node[dot] (5') at (1,2) {};

\draw[black!15] (1) -- (2) -- (3) -- (4) -- (5) -- (6);
\draw[black!15] (1') -- (2') -- (3') -- (4') -- (5') -- (6');

\draw[very thick,darkblue] (2) -- (3');
\draw[very thick,darkblue] (3) -- (2');
\draw[very thick,darkblue] (4) -- (4');
\draw[very thick,darkblue] (5) -- (5');

\draw[thick,darkred] (4) to[bend left=30] (5);
\draw[thick,darkred] (4') to[bend right=30] (5');
\end{tikzpicture}
\right)
=
\begin{tikzpicture}[baseline=1.15cm]
\node[dot] (1) at (0,0) {};
\node[dot] (1') at (1,0) {};
\node[dot] (6) at (0,2.5) {};
\node[dot] (6') at (1,2.5) {};
\node[dot] (2) at (0,0.5) {};
\node[dot] (3) at (0,1) {};
\node[dot] (4) at (0,1.5) {};
\node[dot] (5) at (0,2) {};
\node[dot] (2') at (1,0.75) {};
\node[dot] (4') at (1,1.5) {};
\node[dot] (5') at (1,2) {};

\draw[black!15] (1) -- (2) -- (3) -- (4) -- (5) -- (6);
\draw[black!15] (1') -- (2') -- (4') -- (5') -- (6');

\draw[very thick,darkblue] (3) -- (2');
\draw[thick,->,lightblue] (2) -- (4');
\draw[very thick,darkblue] (4) -- (4');
\draw[very thick,darkblue] (5) -- (5');

\draw[thick,darkred] (4) to[bend left=30] (5);
\draw[thick,darkred] (4') to[bend right=30] (5');
\end{tikzpicture}
-
\begin{tikzpicture}[baseline=1.15cm]
\node[dot] (1) at (0,0) {};
\node[dot] (1') at (1,0) {};
\node[dot] (6) at (0,2.5) {};
\node[dot] (6') at (1,2.5) {};
\node[dot] (2) at (0,0.5) {};
\node[dot] (3) at (0,1) {};
\node[dot] (4) at (0,1.5) {};
\node[dot] (5) at (0,2) {};
\node[dot] (2') at (1,0.75) {};
\node[dot] (4') at (1,1.5) {};
\node[dot] (5') at (1,2) {};

\draw[black!15] (1) -- (2) -- (3) -- (4) -- (5) -- (6);
\draw[black!15] (1') -- (2') -- (4') -- (5') -- (6');

\draw[thick,lightblue,->] (2) -- (1');
\draw[very thick,darkblue] (3) -- (2');
\draw[very thick,darkblue] (4) -- (4');
\draw[very thick,darkblue] (5) -- (5');

\draw[thick,darkred] (4) to[bend left=30] (5);
\draw[thick,darkred] (4') to[bend right=30] (5');
\end{tikzpicture}
\end{equ}
It is then clear from the definition of $\CJ$ that further applications of $\CJ$ cannot 
lead to a full graph, so that these terms contribute by at most $o(\eps)$ by Theorem~\ref{theo:boundepsdelta}
\end{proof}

\begin{theorem}\label{thm:mainRigorous}
Given an initial condition $x_0 \in \R^d$, 
let $\mu_\eps$ be the law of the maximal solution to \eqref{e:SDEfbm},
viewed as an $(\R^d)^\sol$-valued random variable. Then, $\mu_\eps$ converges weakly to
the law $\mu$ of the maximal solution to \eqref{e:limitDiffusion} with $\sigma^2 = c$
where $c$ is given by \eqref{e:intKbar}.
\end{theorem}

\begin{proof}
By continuity of the solution map $\R^d \times \cC_g^\alpha \to (\R^d)^\sol$, it remains to show 
that the solution to the rough differential equation (RDE) driven by $\X$ coincides with that of the Stratonovich SDE
\eqref{e:limitDiffusion}. In view of \cite[Thm~9.1]{Book}, it suffices to show the following.
Assume that we are given an $\R^{m}\otimes \R^{m}$-valued $2\alpha$-Hölder continuous geometric rough path $(\B^{(1)},\B^{(2)})$
additionally satisfying the antisymmetry relations
\begin{equ}[e:shuffle]
\B^{(1)}_{ij} + \B^{(1)}_{ji} = 0\;,\quad 
\B^{(2)}_{ijk\ell} + \B^{(2)}_{jik\ell} =0\;,\quad  \B^{(2)}_{ijk\ell} + \B^{(2)}_{ij\ell k} = 0\;. 
\end{equ}
In our case, $\B$ is given by the Stratonovich lift of $W$, which indeed satisfies the
additional relations \eqref{e:shuffle} as a consequence of the fact that $W_{ji} = - W_{ij}$.

Given such a $\B$, we claim that the fourth order $\R^m$-valued rough path $\X$ given by
\begin{equ}[e:defX]
\X^{(1)} = 0\;,\quad 
\X^{(2)} = \B^{(1)}\;,\quad 
\X^{(3)} = 0\;,\quad 
\X^{(4)} = \B^{(2)}\;,
\end{equ}
then satisfies $\X \in \cC^\alpha_g$. Furthermore, for any collection of $m$ smooth functions 
$V_i \colon \R^d \to \R^d$, the solution to the RDE
\begin{equ}[e:RDE4]
dx(t) = V_i(x)\,d\X_i(t)\;,\qquad x(0) = x_0\;,
\end{equ}
coincides with the solution to the RDE
\begin{equ}[e:RDE2]
dx(t) = \f12 [V_i, V_j](x)\,d\B_{ij}(t)\;,\qquad x(0) = x_0\;.
\end{equ}
The claim of the theorem then follows at once from the fact that the RDE solutions 
to \eqref{e:RDE2} with $\B$ the Stratonovich lift of $W$ do coincide with the 
solutions to the Stratonovich SDE \eqref{e:limitDiffusion} by \cite[Thm~9.1]{Book}.

The proof that solutions to \eqref{e:RDE4} and \eqref{e:RDE2} coincide is folklore 
(see \cite[Sec.~3]{Sussmann} for a pre-rough path statement strongly hinting at it, 
but it certainly goes back even further since these kind of statements motivate the definition
of Lie brackets and form the basis of geometric control theory)
but this higher order version does not seem to appear in the literature, so 
we provide a sketch of proof.
The fact that $\X$ as defined by \eqref{e:defX} satisfies Chen's relations and the required analytic bounds 
is easy to verify for any $2\alpha$-Hölder rough path $\B$ so that, in order to show that $\X \in \cC^\alpha_g$, it remains to verify that the shuffle relations
(see for example \cite[Eq.~2.21]{Book}) hold. This however follows immediately from the shuffle relations for $\B$,
combined with the fact that we enforced \eqref{e:shuffle}.

After unpacking the definitions one finds that solutions to \eqref{e:RDE4} are 
in general given by the fixed point problem
\begin{equs}
x(t) &= x_0 + \lim_{|\CP_t| \to 0} \sum_{[u,v] \in \CP_t} \Big(V_i\X_i^{(1)} + DV_i (V_j)\X_{ji}^{(2)} + D^2V_i \big(V_j, V_k\big)\X_{kji}^{(3)}\\
 &\quad + DV_i\bigl(DV_j ( V_k)\bigr)\X_{kji}^{(3)} + D^3 V_i \big(V_j, V_k, V_\ell\big)\X_{\ell kji}^{(4)} \\
 &\quad + DV_i \bigl(DV_j(DV_k(V_\ell))\bigr)\X_{\ell kji}^{(4)}
 + DV_i \bigl(D^2V_j(V_k,V_\ell)\bigr)\X_{\ell kji}^{(4)}\\
&\quad + D^2V_i \bigl(V_j, D_k(V_\ell)\bigr)\bigl(\X_{j\ell ki}^{(4)}
+\X_{\ell jki}^{(4)}+\X_{\ell kji}^{(4)}\bigr)
 \Big) \;,
\end{equs}
where $\CP_t$ denotes a partition of $[0,t]$ and $|\CP_t|$ is the length of its largest subinterval.
The terms involving $\X_i^{(1)}$ and $\X_i^{(3)}$ vanish by \eqref{e:defX}. The term involving $\X^{(2)}$ can be 
rewritten by the first identity in \eqref{e:shuffle} as
\begin{equs}
\f12(DV_i(x_u) V_j(x_u))&\big(\B_{ji}^{(1)}(u,v) - \B_{ij}^{(1)}(u,v)\big) \\
&= \f12(DV_j(x_u) V_i(x_u) - DV_i(x_u) V_j(x_u))\B_{ij}^{(1)}(u,v)\\
&= \f12 [V_i,V_j](x_u)\B_{ij}^{(1)}(u,v)\;.
\end{equs}
Regarding the four terms involving $\X^{(4)}$, the first one vanishes since $\X_{\ell kji}^{(4)}$
in antisymmetric in $(k,\ell)$ by \eqref{e:shuffle} and \eqref{e:defX} while 
the prefactor is symmetric. The third one
vanishes for an analogous reason. Using again \eqref{e:shuffle} and \eqref{e:defX},
a straightforward calculation shows that the sum of the remaining two terms equals 
\begin{equ}
\f14 D[V_i,V_j] ([V_k,V_\ell])\,\B_{k\ell ij}^{(2)}\;,
\end{equ}
so that one indeed recovers the fixed point problem for \eqref{e:RDE2} as required.
\end{proof}

\begin{remark}
The convergence of course also holds at the level of flows.
\end{remark}

\bibliographystyle{Martin}
\bibliography{Variance}

\end{document}